\documentclass[reqno]{amsart}

\usepackage{graphicx}
\allowdisplaybreaks

\usepackage{amsmath,amsthm,amscd,amssymb,amsfonts, amsbsy}
\usepackage{latexsym}
\usepackage{color}
\usepackage{enumerate}
\usepackage{pxfonts}
\usepackage{marginnote}
\usepackage{environ}

\theoremstyle{plain}
\newtheorem{theorem}{Theorem}[section]
\newtheorem{proposition}[theorem]{Proposition}
\newtheorem{lemma}[theorem]{Lemma}

\theoremstyle{definition}

\theoremstyle{remark}

\DeclareMathOperator *{\di}{div} 

\newcommand{\dist}{\operatorname{dist}}
\newcommand{\diam}{\operatorname{diam}}

\newcommand{\bR}{\mathbb{R}}
\newcommand{\bZ}{\mathbb{Z}}

\newcommand\cL{\mathcal{L}}

\providecommand{\set}[1]{\{#1\}}

\providecommand{\abs}[1]{\lvert#1\rvert}

\providecommand{\norm}[1]{\lVert#1\rVert}

\DeclareMathOperator *{\essosc}{ess\ osc}
\DeclareMathOperator *{\osc}{osc}
\DeclareMathOperator *{\essinf}{ess\ inf}

\begin{document}
\title[Lipschitz regularity of homogenization with continuous coefficients]
{Lipschitz regularity of homogenization with continuous coefficients: Dirichlet problem}

\author[S. Lee]{Sungjin Lee}
\address{Sungjin Lee \\ Aristotle University of Thessaloniki \\ School of Mathematics \\ 541 24 Thessaloniki, Greece}
\email{slee@math.auth.gr}
\thanks{This work was supported (in part) by the Yonsei University Research Fund(Post Doc. Researcher Supporting Program) of 2022(2023) (project no.: 2022-12-0135), and this project is carried out within the framework of the National Recovery and Resilience Plan Greece 2.0, funded by the European Union – NextGenerationEU (Implementation body: HFRI, Project Name: CRITPDE, No. 14952).}


\begin{abstract}
We study uniform Lipschitz regularity estimates for elliptic systems in divergence form with continuous coefficients, based on rapidly oscillating periodic coefficients derived from homogenization theory. We extend a result by Avellaneda and Lin [Comm. Pure Appl. Math. 40 (1987), pp. 803-847] by minimizing all regularity conditions of the given data to integral conditions. We remark that the coefficients of an elliptic operator have Dini mean oscillation, which corresponds to the results of the latest general regularity theory.
\end{abstract}
 
\maketitle

\section{Introduction and main results}
We consider an elliptic operator $\cL_\varepsilon$ in divergence form
\[				
\cL_\varepsilon =-D_i\bigg\{ a_{ij}^{\alpha\beta}\big( \frac{x}{\varepsilon}\big)D_j\bigg\}= -\di \bigg\{ \mathbf{A}\big(\frac{x}{\varepsilon}\big)\nabla\bigg\},\quad \varepsilon>0,\\
\]
where the coefficient matrix $\mathbf{A}:=(a^{\alpha\beta}_{ij}(y))$  with $1\leq i,j\leq n $ and $1\leq \alpha,\beta\leq m$ are symmetric and satisfy the ellipticity condition:
\begin{equation}				\label{ellipticity}
\lambda \abs{\xi}^2 \abs{\eta}^2 \le a^{\alpha\beta}_{ij}\xi_i\xi_j\eta_\alpha \eta_\beta \quad\hbox{and}\quad \norm{A}_{L^\infty(\bR^n)}\leq \Lambda
\end{equation}
for any $y,\xi\in \bR^n$, $\eta\in\bR^m$ , some $\lambda$, $\Lambda$ are positive constants and $\Omega$ is a bounded domain in $\bR^n$. Here and throughout the paper, we apply the summation convention over repeated indices.

We assume that the coefficient matrix $\mathbf{A}(y)$ is real and $1$-periodic, that is, we have
\begin{equation} \label{perodicity}
\mathbf{A}(y+z)=\mathbf{A}(y)
\end{equation}
for $y\in\bR^n$ and $z\in \bZ^n$.

For $x\in \bR^n$ and $r>0$, we denote by $B(x,r)$ the open ball with radius $r$ centered at $x$, and write $\Omega(x,r):=\Omega \cap B(x,r)$.
We denote
\[
\omega_{\mathbf A}(r, x):= \fint_{\Omega(x,r)} \,\abs{\mathbf{A}(y)-\bar {\mathbf A}^{\Omega(x,r)}}\,dy, \quad \text{ where } \;\bar{\mathbf A}^{\Omega(x,r)} :=\fint_{\Omega(x,r)} \mathbf{A},
\]
and, for a subset $D$ of $\bR^n$, we write
\[
\omega_{\mathbf A}(r, D):= \sup_{x \in D} \omega_{\mathbf A}(r, x) \quad \text{and}\quad \omega_{\mathbf A}(r)=\omega_{\mathbf A}(r, \bR^n).
\]
We say that $\mathbf{A}$ is of \emph{Dini mean oscillation} in $\bR^n$ and write $\mathbf{A}\in \textsf{DMO}(\bR^n)$ if $\omega_{\mathbf A}(r)$ satisfies the Dini's condition;
\begin{equation*}		
\int_0^1 \frac{\omega_{\mathbf A}(t)}t \,dt <+\infty.
\end{equation*}
 
For a locally integrable function $f$, we denote the modulus of continuity
\begin{equation*}				
\varrho_{f} (r, x)= \sup \set{\abs{  f(x)-f(x_0)}: x_0 \in \Omega, \;\abs{x-x_0} \le r}
\end{equation*}
and, for a subset $D$ of $\bR^n$, we write
\[
\varrho_{f}(r,D):=\sup_{x\in D }\varrho_{f}(x)\quad \text{and}\quad \varrho_{f}(r)= \varrho_{f}(r, \overline{\Omega}).
\]
We shall say that $f$ is (uniformly) Dini continuous in $\Omega$ if its modulus of continuity satisfies the Dini's condition;
\[
 \int_0^1 \frac{\varrho_{f}(t)}t \,dt <+\infty.
\]

Before proceeding further, we now clarify the relationship of the regularity assumption to precisely understand the subsequent steps.
It is obvious that if $f$ is Dini continuous in $\Omega$, then $f$ is of Dini mean oscillation in $\Omega$.
However, a function of Dini mean oscillation is noticeably less restrictive then Dini continuous function; see \cite{DK17} for an example.

A few remarks are in order. Various topics of homogenization theory are covered in many books (see, e.g., \cite{ShenBook,LionBook,DoinaBook} and references therein). A well-known result in the regularity theory of homogenization is the work of Avellaneda and Lin \cite{AL87}. They introduced a compactness method, which originated from the calculus of variation, and proved that if the solution $u_\varepsilon\in W^{1,2}(\Omega;\bR^m)$ of the Dirichlet problem
\[
\cL_\varepsilon u_\varepsilon= F\;\text{ in }\;\Omega,\quad u_\varepsilon=g\;\text{ on }\;\partial \Omega,
\]
with coefficients that are H\"{o}lder continuous, then $u_\varepsilon$ satisfies H\"{o}lder regularity estimates and the following Lipschitz regularity estimates;
\begin{equation*} 
\norm{\nabla u_\varepsilon}_{L^\infty(\Omega)}\leq C\big(\norm{F}_{L^p(\Omega)} +   
\norm{\nabla g}_{C^{1,\mu}(\partial\Omega)}  \big).
\end{equation*}

As a matter of fact, the H\"{o}lder regularity can be established using the real-variable method to obtain $W^{1,p}$ estimates when the coefficients are sufficiently minimal in the vanishing mean oscillation setting; see Lemma \ref{lemmaholder}. However, if the coefficients are uniformly continuous, then it is well-known that the solution $u_\varepsilon$ may fail to achieve Lipschitz regularity. After their well-known results, most Lipschitz regularity results are based on the assumption that the coefficients are H\"{o}lder continuous. To mention just a few, under the same assumption, the Lipschitz regularity for Neumann problem, parabolic and Stokes's system, and considering lower-order terms have been studied; see \cite{KLS13,GS15,GS151,GX17,Xu16,WZ23}.

In another studied case of regularity condition of the given data, Armstrong and Shen \cite{AS16} established Lipschitz regularity in almost-periodic coefficients with suitable coefficients, which can be deduced up to Dini continuous; see \cite[Lemma 4.3]{AS16} with example in \cite{DK17}. Shen \cite{Shen17} also established Lipschitz regularity on the $C^{1,\alpha}$ domain by adding smoothness condition to the coefficients, which can be deduced up to Dini mean oscillation. Recently, results have also been proven where both the coefficients and the boundary data satisfy Dini continuity; see \cite{DLZ21, DL21}.

In this paper, we study the minimal (weakest) regularity conditions for Lipschitz regularity not only on the coefficients but also on the domain, boundary data, and external data. Therefore, we will establish interior Lipschitz estimates of the Dirichlet problem when the coefficients and external data are in Dini mean oscillation, as well as boundary Lipschitz regularity of the Dirichlet problem when the coefficients are in Dini mean oscillation and the other data satisfy Dini continuity.

Before stating our main theorem precisely, we denote function as $C^{1,Dini}$ if it is a $C^1$ function whose first derivatives are Dini continuous.
Now, we state the main theorems. 
 
\begin{theorem} \label{thm00}
Assume the coefficients $\mathbf{A}=(a^{ij})$ of the operator $\cL_\varepsilon$ satisfy the  condition \eqref{ellipticity}, \eqref{perodicity} and are of Dini mean oscillation in $\bR^n$. Let $u_\varepsilon\in W^{1,2}(B(x_0,2R);\bR^m)$ be a weak solution of
\[ 
\cL_\varepsilon u_\varepsilon=\di f \;\text{ in }\;B(x_0,2R)
\]
where $f:\Omega\rightarrow\bR^{n\times m}$ are of Dini mean oscillation and $0<R\leq\frac{1}{2}R_0$. Then, we have
\begin{equation*} 
\norm{\nabla u_\varepsilon}_{L^\infty(B(x_0,R))}\leq C\Bigg\{\frac{1}{R}    \bigg(\fint_{B(x_0,2R)}\abs{  u_\varepsilon}^2 \bigg)^{1/2}  +  \int^R_0\frac{\omega_f(t,B(0,2R))}{t}dt  \Bigg\},
\end{equation*}
where $C=C(n,m,\lambda,\Lambda,\omega_{\mathbf{A}},R_0)$ .
\end{theorem}
\begin{theorem}					\label{thm01}
Let $\Omega$ be a bounded $C^{1,Dini}$ domain in $\bR^n$ .
Assume the coefficients $\mathbf{A}=(a^{ij})$ of the operator $\cL_\varepsilon$ satisfy the  condition \eqref{ellipticity}, \eqref{perodicity} and are of Dini mean oscillation in $\bR^n$. Let $u_\varepsilon\in W^{1,2}(\Omega;\bR^m)$ be a weak solution of
\begin{equation} \label{main_eq}
\cL_\varepsilon u_\varepsilon=\di f +F\;\text{ in }\;\Omega,\quad u_\varepsilon=g\;\text{ on }\;\partial \Omega,
\end{equation} 
where $f:\Omega\rightarrow\bR^{n\times m}$ are of Dini continuous and $F\in L^p(\Omega;\bR^m)$ for $p>n$ and  $g\in C^{1,Dini}(\overline{ \Omega};\bR^m)$. Then, we have
\begin{equation*} 
\norm{\nabla u_\varepsilon}_{L^\infty(\Omega)}\leq C\big(\norm{F}_{L^p(\Omega)} +  
\norm{\nabla g}_{L^\infty(\partial\Omega)}+\int^1_0\frac{\varrho_f(t)+\varrho_{\nabla  g}(t)}{t}dt  \big),
\end{equation*}
where $C=C(n,m,\lambda,\Lambda,p,\Omega,\omega_{\mathbf{A}})$.
\end{theorem}

We provide a brief description of the proof used in the theorem and remark on its novelties.
To show Theorem \ref{thm00} and Theorem \ref{thm01}, we adapt the compactness method, which involves three successive steps. This method requires more delicate iterative estimates, some estimates considering boundary correctors, and controlling the given data relative to each other. It is not easier than previous results \cite{AS162,AS16,Shen17}. However, it is useful for studying the minimal condition for Lipschitz regularity and distinctly understanding the classification of each regularity condition. Through this approach, we extend all regularity conditions of the given data to integral conditions (e.g. Dini's condition) for minimization purposes, rather than pointwise conditions  (e.g. H\"{o}lder's continuous).  In particular, with Dini mean oscillation coefficients $\mathbf{A}$, we derive new results: when $f$  is Dini mean oscillation in interior regularity and $\Omega$ is a $C^{1,Dini}$ domain in boundary regularity. Our results correspond exactly to the latest general regularity theory for $\cL_1$; see, \cite{DK17, DEK18,DLK20}.

 Finally, the organization of the paper is as follows: In Section \ref{section2}, we state some preliminary lemmas. The proofs of Theorems \ref{thm00}, boundary Lipschitz estimates and Theorem \ref{thm01} are given in Sections \ref{section3}, \ref{section4}, and \ref{section5}.
\section{Preliminaries} \label{section2}
In this section, we present some results which will be used in the proofs of Theorem \ref{thm00} and Theorem \ref{thm01}.

From well known results; see \cite{LionBook,DoinaBook,ShenBook}, we consider the homogenized operator $\cL_{\hat{\mathbf{A}},0}$ that if $F\in L^2(\Omega;\bR^m)$, $f\in L^2(\Omega;\bR^{n\times m})$ and $g\in H^{1/2}(\partial\Omega;\bR^m)$ then, as $\varepsilon \rightarrow 0$, $u_\varepsilon$ converges weakly in $H^1(\Omega;\bR^m)$ to $u_0$ that solves the \emph{homogenized (effective) problem};
\[
\cL_{\hat{\mathbf{A}},0}u_0^\alpha: =-D_i(\hat{\mathbf{A}}_{ij}^{\alpha\beta}D_ju_0^\beta)=D_i f^\alpha_i +F^\alpha \;\text{ in }\;\Omega,\quad u_0^\alpha=g^\alpha\;\text{ on }\;\partial \Omega,
\]
where constants matrix $\hat{\mathbf{A}}=(\hat{a}^{\alpha\beta}_{ij})$ with $1\leq i,j\leq n $, $1\leq \alpha,\beta\leq m$,
\[
\hat{a}^{\alpha\beta}_{ij}=\fint_{[0,1)^n}  a^{\alpha\beta}_{ij}+a^{\alpha\gamma}_{ik}D_{y_k}\chi_j^{\gamma\beta} dy
\]
and the  $1$-periodic matrix $\chi=(\chi^{\beta}_i)=(\chi^{\alpha\beta}_i)$, with $1\leq i\leq n$, $1\leq \alpha,\beta\leq m$, the \emph{matrix of (first-order) correctors} is defined as the solution of the \emph{cell problem};
\begin{equation}		\label{corrector}
-D_i(\mathbf{A}_{ik}^{\alpha\gamma}(y)D_k \chi_j^{\gamma\beta}(y))=D_i\mathbf{A}_{ij}^{\alpha\beta}(y)      \; \text{in}\;\bR^n\; \quad \text{and} \quad \int_{[0,1)^n}\chi^{\beta}_j dy=0.
\end{equation}
To simplify the notation, we re-write the above equation as 
\[
\cL_1(\chi^\beta_j+I^\beta_j)=0\;\text{in}\;\bR^n,
\]
where $I^\beta_j=I^\beta_j(y)=y_je^\beta$ is defined as $e^{\beta}=(0,\ldots,1,\ldots,0)$ with 1 in the $\beta$th position. \

We recall that $\mathbf{A}$ belong to $\textsf{VMO}_{\mathbf{A}}$ if and only if $\lim_{r\rightarrow 0} \omega_{\mathbf{A}}(r)=0$ (see, e.g., \cite{Krylov}) and thus $\textsf{VMO}_{\mathbf{A}}$ contains $\textsf{DMO}_{\mathbf{A}}$. We state for precisely that the following results are the uniform H\"{o}lder estimates for $\cL_\varepsilon$, which have been studied by many authors, with some work being \cite{AL87,ShenBook,WZ23,Xu16}.

\begin{lemma} \label{lemmaholder}
Let $\Omega$ be a bounded $C^{1}$ domain in $\bR^n$ .
Assume the coefficients $\mathbf{A}=(a^{ij})$ of the operator $\cL_\varepsilon$ satisfy the  condition \eqref{ellipticity}, \eqref{perodicity} and are of vanishing mean oscillation in $\bR^n$ ($\textsf{VMO}_{\mathbf{A}}$). Let $u_\varepsilon\in W^{1,2}(\Omega;\bR^m)$ be a weak solution of
\[ 
\cL_\varepsilon u_\varepsilon= \di f \;\text{ in }\;\Omega,\quad u_\varepsilon=g\;\text{ on }\;\partial \Omega,
\] 
where $f\in L^p(\Omega;\bR^{n\times m})$ for $p>n$ and    $g\in C^{0,\mu}(\partial \Omega;\bR^m)$. Then, we have
\[   
\norm{ u_\varepsilon}_{C^\mu(\Omega)}\leq C\big(\norm{f}_{L^p(\Omega)}+\norm{g}_{C^{0,\mu}(\partial\Omega)}  \big),
\]
where $C=C(n,m,\lambda,\Lambda,p,\Omega,\textsf{VMO}_{\mathbf{A}})$ and $\mu=1-n/p$.
\end{lemma}

Notice that if $\mathbf{A}$ is of Dini mean oscillation, then there is a modification   $\bar{\mathbf A}$ of $\mathbf A$ that is uniformly continuous with its modulus of continuity controlled by $\omega_\mathbf{A}$; see \cite[Appendix]{HK20} for the proof as well as \cite[Lemma 2.7.]{DK17}. Therefore, without loss of generality, we shall assume the following lemma.

\begin{lemma} \label{lemmaC1}
Let $\Omega$ be a domain. Suppose that $f\in L^1_{loc}(\overline{\Omega})$ is of Dini mean oscillation in $\Omega$, Then for $0<\kappa<1$, we have
\[
\varrho_f(r)+\sum^\infty_{i=0}\omega_f(\kappa^ir)\leq C_1 \int^r_0\frac{\omega_f(t)}{t}dt.
\]
In particular, if $f$ is Dini continuous, Then for $0<\kappa<1$, we also have
\[
\sum^\infty_{i=0}\varrho_f(\kappa^ir)\leq C_1 \int^r_0\frac{\varrho_f(t)}{t}dt.
\]
The constants $C_1$ depend at most on $n$ and $\Omega$.
\end{lemma}
 
Throughout the remaining paper, we can assume the following  without loss of generality (see \cite[Lemma 2.10.]{DLK20});
\begin{equation} \label{dec_omega}
\omega_\bullet (t)/t^\beta,\; \varrho_\bullet(t)/t^\beta \;\text{are decreasing}\;\text{for}\;\beta\in(0,1].
\end{equation}

For a more comprehensive understanding of the detailed computations and various relationships under the Dini mean oscillation condition, readers may refer to \cite{DK17,DEK18, DLK20,HK20, MMPT23}, as well as the references therein, which provide detailed explanations.

\section{Proof of Interior Lipschitz estimates} \label{section3}
 
 In this section, we prove the interior Lipschitz estimate with Dini mean oscillation inhomogeneous terms using by compactness method.

\begin{lemma}[\emph{One-step improvement}]  \label{lemma01}
There exist constants $\varepsilon_0\in(0,1/4)$ and $\theta\in(0,1/2)$, depending only $n$, $m$, $\lambda$, $\Lambda$, and $\omega_\mathbf{A}$ such that
if $u_\varepsilon$, $f$ satisfy
\[
\bigg(\fint_{B(0,2)}\abs{u_\varepsilon}^2 \bigg)^{1/2}\leq 1,\quad \norm{f}_{L^\infty(B(0,2))}\leq \varepsilon_0,
\]
with \eqref{ellipticity}, \eqref{perodicity}, and
\[
\cL_\varepsilon u_\varepsilon= \di f \; \text{in}\; B(0,2),
\]
then, for any $0<\varepsilon\leq\varepsilon_0$, 
\begin{equation} \label{eq5111thu}
\begin{split}
&\bigg(\fint_{B(0,\theta)} \bigg| u_\varepsilon(x)-\big( I^\beta_j(x)+\varepsilon \chi^\beta_j(x/\varepsilon) \big) \overline{D_ju^\beta_\varepsilon}^{B(0,\theta)}  \\ &\qquad -  \overline{ \bigg( u_\varepsilon-\big( I^\beta_j(x)+\varepsilon \chi^\beta_j(x/\varepsilon) \big) \overline{D_ju^\beta_\varepsilon}^{B(0,\theta)}  \bigg) }^{B(0,\theta)}     \bigg| ^2 dx \bigg)^{1/2}   
\leq \theta\omega_\mathbf{A}(\theta,B(0,2)).
\end{split}
\end{equation}
\end{lemma}
\begin{proof}
Suppose that \eqref{eq5111thu} is not true. Then there exists sequence $\{\varepsilon_l\}\subset (0,1/4)$, $\{ \mathbf{A}_l\}\subset \mathsf{DMO}(\bR^n)$ satisfying \eqref{ellipticity} and \eqref{perodicity}, $\{ f_l \}  \subset L^\infty(B(0,2);\bR^{n\times m})$, and $\{  u_{\varepsilon_l}\}\subset W^{1,2}(B(0,2);\bR^m)$ such that $\varepsilon_l\rightarrow 0$,
\begin{align}
\bigg(\fint_{B(0,2)}\abs{u_{\varepsilon_l}}^2 \bigg)^{1/2}\leq 1,\quad \norm{f_l}_{L^\infty((B(0,2))}\leq \varepsilon_l,  \label{eq5112thu}\\
\cL_{\mathbf{A}_l,\varepsilon_l}u_{\varepsilon_l}:=-\di (\mathbf{A}_l (x/\varepsilon_l)\nabla u_{\varepsilon_l})=\di f_l\;\text{in}\;B(0,2),  \label{eq5113thu}
\end{align}
and
\begin{equation} \label{eq5114thu}  
\begin{split}
&\bigg(\fint_{B(0,\theta)} \bigg| u_{\varepsilon_l}(x)-\big( I^\beta_j(x)+\varepsilon \chi^\beta_{\mathbf{A}_l,j}(x/\varepsilon) \big) \overline{D_ju^\beta_{\varepsilon_l}}^{B(0,\theta)}  \\ &\qquad -  \overline{ \bigg( u_{\varepsilon_l}-\big( I^\beta_j(x)+\varepsilon \chi^\beta_{\mathbf{A}_l,j}(x/\varepsilon) \big) \overline{D_ju^\beta_{\varepsilon_l}}^{B(0,\theta)}  \bigg) }^{B(0,\theta)}     \bigg| ^2 dx \bigg)^{1/2}   
>\theta\omega_\mathbf{A}(\theta,B(0,2)).
\end{split}
\end{equation}
where $\chi^\beta_{\mathbf{A}_l,j}$ denotes the correctors for $\mathbf{A}_l$.
Observe that, by \eqref{eq5112thu}, \eqref{eq5113thu} and Caccioppoli's inequality, $\{u_{\varepsilon_l}\}$ is uniformly bounded in $W^{1,2}(B(0,1))$. By passing to subsequences, we see that, as $l\rightarrow \infty$, 
\begin{align*}
&u_{\varepsilon_l} \rightarrow u_0\quad\;\text{in}\; L^2(B(0,1)),\qquad u_{\varepsilon_l} \rightharpoonup  u_0\quad\;\text{in}\; L^2(B(0,2)),\\
&D_j u_{\varepsilon_l} \rightharpoonup D_j u_0\quad\;\text{in}\; L^2(B(0,1)),\\
&f_l\rightarrow 0\quad\; \text{in}\; L^\infty(B(0,2)),\qquad \hat{\mathbf{A}_l}\rightarrow \mathbf{A}^0,
\end{align*}
and
\begin{equation}  \label{eq5116thu}
\cL_{\mathbf{A}^0,0}u_0=0\;\text{in}\;B(0,1),
\end{equation}

where $\hat{\mathbf{A}_l}$ denotes the homogenized coefficients for $\mathbf{A}_l$ and $\mathbf{A}^0$ is some constant matrix satisfying \eqref{ellipticity}, \eqref{perodicity}. We now let $l\rightarrow \infty$, in \eqref{eq5112thu} and \eqref{eq5114thu}. This leads to
\begin{equation}  \label{eq5115thu}
\bigg(\fint_{B(0,2)}\abs{u_0}^2 \bigg)^{1/2}\leq 1
\end{equation}
and 
\[  
\bigg(\fint_{B(0,\theta)} \bigg| u_0(x)-x_j \overline{D_ju_0}^{B(0,\theta)}  -  \overline{  u_0 }^{B(0,\theta)}    \bigg| ^2 dx \bigg)^{1/2}   
\geq \theta\omega_\mathbf{A}(\theta,B(0,2)).
\]
Here, we have used the fact that $\chi^\beta_{\mathbf{A}_l,j}$ is bounded in $L^2([0,1)^n;\bR^{n\times m^2})$. 
From interior $C^2$ estimates of $\eqref{eq5116thu}$ with $\eqref{eq5115thu}$, for any $\theta\in(0,1/2)$, we have that
\[  
\bigg(\fint_{B(0,\theta)} \bigg| u_0(x)-x_j \overline{D_ju_0}^{B(0,\theta)}  -  \overline{  u_0 }^{B(0,\theta)}    \bigg| ^2 dx \bigg)^{1/2} \leq C_2 \theta^2 
\]
where $C_2=C_2(n,m,\lambda,\Lambda)$. Then, for some $0<\varepsilon_0<\frac{1}{4}\min\{1,1/C_1\}$ where the constant $C_1$ comes from Lemma \ref{lemmaC1}, we can choose sufficiently small $\theta$ such that
\begin{equation}	\label{eq5231tue}
C(\lambda,n,m)C_2\theta^{1-\beta}\leq \frac{1}{2} \quad\text{and}\quad\theta+\int^{\theta}_0\frac{\omega_{\mathbf{A}}(t,B(0,2))}{t}dt\leq \varepsilon_0,   
\end{equation}
where $\theta<1/2$, 
and $0<\beta<1$ with $\eqref{dec_omega}$. It then follows that 
\[  
\bigg(\fint_{B(0,\theta)} \bigg| u_0(x)-x_j \overline{D_ju_0}^{B(0,\theta)}  -  \overline{  u_0 }^{B(0,\theta)}    \bigg| ^2 dx \bigg)^{1/2} \leq \frac{1}{2}\theta \omega_\mathbf{A}(\theta,B(0,2)),
\]
which is in contradiction. Therefore, $\eqref{eq5111thu}$ holds for some $0<\varepsilon_0<\frac{1}{4}\min\{1,\frac{1}{C_1}\}$.
\end{proof}

\begin{lemma}[\emph{Iteration}]  \label{lemma02}
Let $\varepsilon_0$ and $\theta$ be the constants given by Lemma \ref{lemma01}. 
Suppose the functions $u_\varepsilon$, $f$ satisfy
\[
\bigg(\fint_{B(0,2)}\abs{u_\varepsilon}^2 \bigg)^{1/2}\leq 1,\quad C_1\int^1_0\frac{ \omega_f(t,B(0,2))}{t}dt\leq \varepsilon_0, 
\]
and
\[
\cL_\varepsilon u_\varepsilon=\di f \; \text{in}\; B(0,2),
\]
then, for some $k\geq 1$ with $0<\varepsilon<\theta^{k-1}\varepsilon_0$, there exists constants $\mathbf{B}(\varepsilon,l)=\big(\mathbf{B}^\beta_j(\varepsilon,l)\big)\in \bR^{n\times m}$ for $1\leq l\leq k$, such that
\begin{equation}  \label{eq5151mon}
\begin{split}
&\abs{\mathbf{B}(\varepsilon,l)}\leq C\big(1+\sum^{l-1}_{i=0}\int^{\theta^i}_0\frac{\omega_\mathbf{A}(t,B(0,2))+C_1\omega_f(t,B(0,2))}{t}dt\big),\\
&\abs{\mathbf{B}(\varepsilon,l+1)-\mathbf{B}(\varepsilon,l)}\leq C\int^{\theta^l}_0\frac{\omega_\mathbf{A}(t,B(0,2))+C_1\omega_f(t,B(0,2))}{t}dt,
\end{split}
\end{equation}
and
\begin{align}   
\bigg(\fint_{B(0,\theta^l)} \bigg| u_\varepsilon(x)&-\big( I^\beta_j(x)+\varepsilon \chi^\beta_j(x/\varepsilon) \big) \mathbf{B}_j^\beta(\varepsilon,l) \nonumber \\ &\qquad\qquad -  \overline{ \bigg( u_\varepsilon-\big( I^\beta_j(x)+\varepsilon \chi^\beta_j(x/\varepsilon) \big)\mathbf{B}_j^\beta(\varepsilon,l)  \bigg) }^{B(0,\theta^l)}     \bigg| ^2 dx \bigg)^{1/2}   \nonumber
\\ &\leq  \frac{\theta^l }{\varepsilon_0}\int^{\theta^l}_0\frac{\omega_\mathbf{A}(t,B(0,2))+C_1\omega_f(t,B(0,2))}{t}dt.\label{eq5152mon}
\end{align}
where $C=C(n,m,\lambda,\Lambda,\omega_\mathbf{A})$.
\end{lemma}
\begin{proof}
Without loss of generality, we can write 
\[
\cL_\varepsilon u_\varepsilon=\di (f-f(0)) \; \text{in}\; B(0,2).
\]
We prove the lemma by induction on $k$. First, the case $k=1$ hold from Lemma \ref{lemma01} with 
\begin{align*}
\norm{f-f(0)}_{L^\infty(B(0,2))}&
\leq C_1 \int^1_0\frac{\omega_f(t,B(0,2))}{t}dt\leq \varepsilon_0,\quad \mathbf{B}^\beta_j(\varepsilon,1)=\overline{D_ju^\beta_\varepsilon}^{B(0,\theta)},\\
\theta\omega_\mathbf{A}(\theta,B(0,2))&\leq \theta C_1\int^{\theta }_0\frac{\omega_\mathbf{A}(t,B(0,2)) }{t}dt\\&\leq \frac{\theta}{\varepsilon_0} \int^{\theta}_0\frac{\omega_\mathbf{A}(t,B(0,2))+C_1\omega_f(t,B(0,2)) }{t}dt
\end{align*}
and \eqref{eq5151mon} be followed from Caccioppoli's inequality with constants $\theta$, $\varepsilon_0$.

Suppose there exists constants $B(\varepsilon,l)$ such that $\eqref{eq5151mon}$, $\eqref{eq5152mon}$ for all integers up to some $l$, where $1\leq l\leq k-1$ with $\varepsilon<\theta^{k-1}\varepsilon_0$. For $x\in B(0,2)$, we consider the function 
\begin{align*}
&w_\varepsilon(x)=\\&\frac{u_\varepsilon(\theta^lx)-\big( I^\beta_j(\theta^lx)+\varepsilon \chi^\beta_j(\theta^l x/\varepsilon) \big) \mathbf{B}_j^\beta(\varepsilon,l)-  \overline{ \bigg( u_\varepsilon-\big( I^\beta_j(x)+\varepsilon \chi^\beta_j(x/\varepsilon) \big) \mathbf{B}_j^\beta(\varepsilon,l) \bigg) }^{B(0,\theta^l)} }{\theta^l \varepsilon^{-1}_0\int^{\theta^l}_0\frac{\omega_\mathbf{A}(t,B(0,2))+C_1\omega_f(t,B(0,2))}{t}dt}.
\end{align*}
By the scaling property and \eqref{corrector}, we have
\begin{equation}	\label{eq5161tue}
\cL_{\varepsilon/\theta^l} w_\varepsilon=\di \tilde{f}  
\end{equation}
where $\tilde{f}  =\varepsilon_0 (f(\theta^lx)-f(0))/(\int^{\theta^l}_0\frac{\omega_\mathbf{A}(t,B(0,2))+C_1\omega_f(t,B(0,2))}{t}dt)$ 
and observe that
\begin{align*}
C_1\int^1_0 \frac{\omega_{\tilde{f}  }(t,B(0,2))}{t}dt\leq \frac{\varepsilon_0C_1}{\int^{\theta^l}_0\frac{\omega_\mathbf{A}(t,B(0,2))+C_1\omega_f(t,B(0,2))}{t}dt} \int^{1}_0\frac{\omega_f(\theta^l t,B(0,2))}{t}dt\leq \varepsilon_0
\end{align*}
 
By the induction hypothesis with \eqref{eq5152mon}, we obtain
\begin{equation}	\label{eq5163tue}
\bigg(\fint_{B(0,2)}\abs{w_\varepsilon}^2 \bigg)^{1/2}\leq 1
\end{equation}
Applying Lemma \ref{lemma01} with \eqref{eq5161tue}, $\varepsilon\theta^{-l}\leq \varepsilon \theta^{-k+1}<\varepsilon_0$, and $\norm{\tilde{f}  }_{L^\infty(B(0,2))}\leq \varepsilon_0$, we have
\begin{equation*}  
\begin{split}
&\bigg(\fint_{B(0,\theta)} \bigg| w_\varepsilon(x)-\big( I^\beta_j(x)+\frac{\varepsilon}{\theta^l} \chi^\beta_j(\theta^l x/\varepsilon) \big) \overline{D_jw^\beta_\varepsilon}^{B(0,\theta)}  \\ &\qquad -  \overline{ \bigg( w_\varepsilon-\big( I^\beta_j(x)+\frac{\varepsilon}{\theta^l} \chi^\beta_j(\theta^l x/\varepsilon) \big) \overline{D_jw^\beta_\varepsilon}^{B(0,\theta)}  \bigg) }^{B(0,\theta)}     \bigg| ^2 dx \bigg)^{1/2}   
\leq \theta\omega_{\tilde{\mathbf{A} } }(\theta,B(0,2)).
\end{split}
\end{equation*}
where $\tilde{\mathbf{A} } =\mathbf{A}(\theta^lx)$. 
Using the definition of $w_\varepsilon$ and scaling, we rewrite that
\begin{equation*}
\begin{split}
\bigg(\fint_{B(0,\theta^{l+1})} \bigg| u_\varepsilon(x)&-\big( I^\beta_j(x)+\varepsilon \chi^\beta_j(x/\varepsilon) \big) \mathbf{B}_j^\beta(\varepsilon,l+1) \\&\qquad\qquad  -  \overline{ \bigg( u_\varepsilon-\big( I^\beta_j(x)+\varepsilon \chi^\beta_j(x/\varepsilon) \big) \mathbf{B}_j^\beta(\varepsilon,l+1)  \bigg) }^{B(0,\theta^{l+1})}     \bigg| ^2 dx \bigg)^{1/2} \\  
&\leq  \frac{  \theta^{l+1}\omega_{\tilde{\mathbf{A} } }(\theta,B(0,2))}{\varepsilon_0}\int^{\theta^l}_0\frac{\omega_\mathbf{A}(t,B(0,2))+C_1\omega_f(t,B(0,2))}{t}dt,
\end{split}
\end{equation*}
where 
\[
\mathbf{B}_j^\beta(\varepsilon,l+1) =\mathbf{B}_j^\beta(\varepsilon,l)+\overline{D_jw^\beta_\varepsilon}^{B(0,\theta)}\int^{\theta^l}_0\frac{\omega_\mathbf{A}(t,B(0,2))+C_1\omega_f(t,B(0,2))}{t}dt
\]
By definition of $\omega$, scaling, aussmption of $f$ and \eqref{eq5231tue}, we observe that
\begin{align*}
&\frac{  \theta^{l+1}\omega_{\tilde{\mathbf{A} } }(\theta,B(0,2))}{\varepsilon_0}\int^{\theta^l}_0\frac{\omega_\mathbf{A}(t,B(0,2))+C_1\omega_f(t,B(0,2))}{t}dt\\
&\leq \theta^{l+1}\frac{C_1}{\varepsilon_0}(\varepsilon_0+\varepsilon_0)\int^{\theta^{l+1}}_0\frac{\omega_\mathbf{A}(t,B(0,2))+C_1\omega_f(t,B(0,2))}{t}dt\\
&\leq \frac{  \theta^{l+1} }{\varepsilon_0}\int^{\theta^{l+1}}_0\frac{\omega_\mathbf{A}(t,B(0,2))+C_1 \omega_f(t,B(0,2))}{t}dt.
\end{align*}
By Caccioppoli's inequality with constants $\theta$, $\varepsilon_0$, \eqref{eq5163tue}, and
\[
\norm{\tilde{f} }_{L^\infty(B(0,2))}\leq \varepsilon_0,
\]
we obtain
\[
\abs{\overline{D_jw^\beta_\varepsilon}^{B(0,\theta)}}\leq C,
\]
where $C=C(n,m,\lambda,\Lambda)$. 
By combining all of them, we conclude \eqref{eq5151mon}, \eqref{eq5152mon}. This completes the induction.
\end{proof}

We now turn to interior estimates by using the blow-up method, the last step of the compactness method, considering the sensitive term $\mathbf{B}(\varepsilon,\cdot)$ with respect to $\varepsilon$.

\begin{lemma}[\emph{Blow-up method}] \label{lemma03}
Assume the coefficients $\mathbf{A}=(a^{ij})$ of the operator $\cL_\varepsilon$ satisfy the  condition \eqref{ellipticity}, \eqref{perodicity} and are of Dini mean oscillation in $\bR^n$. Let $u_\varepsilon\in W^{1,2}(B(0,2);\bR^m)$ be a weak solution of
\[ 
\cL_\varepsilon u_\varepsilon=\di f \;\text{ in }\;B(0,2)
\]
where $f:\Omega\rightarrow\bR^{n\times m}$ are of Dini mean oscillation. Then, we have
 
\begin{equation}  \label{lip_esti}
\norm{\nabla u_\varepsilon}_{L^\infty(B(0,1))}\leq C\Bigg\{   \bigg(\fint_{B(0,2)}\abs{  u_\varepsilon}^2 \bigg)^{1/2}  +\int^1_0\frac{\omega_f(t,B(0,2))}{t}dt  \Bigg\},
\end{equation}
where $C=C(n,m,\lambda,\Lambda,\omega_{\mathbf{A}})$ .
\end{lemma}
\begin{proof}
We may assume that $0<\varepsilon<\varepsilon_0\theta$, where $\varepsilon_0$, $\theta$ are the constants given by Lemmae \ref{lemma01}.  Because in the case of $\varepsilon_0\theta\leq \varepsilon$, it is clearly by $C^1$ estimates with Dini mean oscillation of $\mathbf{A}$ (See, \cite[Theorem 1.5.]{DK17}) and the observation that, for $\varepsilon\geq\varepsilon_0\theta$,
\begin{equation*}  
\int^1_0\frac{\omega_{\mathbf{A}(\cdot/\varepsilon)}(t,B(0,2))}{t}dt=\int^1_0\frac{\omega_\mathbf{A}(t/\varepsilon,B(0,2))}{t}dt\leq \int^{\frac{1}{\varepsilon_0\theta}}_0\frac{\omega_\mathbf{A}(t,B(0,2))}{t}dt<+\infty.
\end{equation*}
Now, we choose $k\geq 2$ such that
\begin{equation}	\label{eq5232tue}
\varepsilon_0\theta^k\leq \varepsilon<\varepsilon_0\theta^{k-1},
\end{equation}
and define
\[
J=\bigg(\fint_{B(0,2)}\abs{  u_\varepsilon}^2 \bigg)^{1/2} +\frac{C_1}{\varepsilon_0}\int^1_0\frac{\omega_f(t,B(0,2))}{t}dt  .
\]
For $x\in B(0,2)$, let 
\[
v_\varepsilon(x)=J^{-1}u_\varepsilon(x)\quad\text{and}\quad h(x)=J^{-1}f(x).
\]
It is easy to check that
\[
\bigg(\fint_{B(0,2)}\abs{  v_\varepsilon}^2 \bigg)^{1/2}\leq 1,\;C_1\int^1_0\frac{\omega_h(t,B(0,2))}{t}dt \leq \varepsilon_0,\;\text{and}\;\cL_\varepsilon v_\varepsilon=\di h \;\text{ in }\;B(0,2).
\]
Using Lemma \ref{lemma02}, we get
\begin{equation*}
\begin{split}
\bigg(\fint_{B(0,\theta^k)} \bigg| v_\varepsilon(x)&-\big( I^\beta_j(x)+\varepsilon \chi^\beta_j(x/\varepsilon) \big) \mathbf{B}_j^\beta(\varepsilon,k)  \\ &\qquad -  \overline{ \big( v_\varepsilon-\big( I^\beta_j(x)+\varepsilon \chi^\beta_j(x/\varepsilon) \big) \mathbf{B}_j^\beta(\varepsilon,k)  \big) }^{B(0,\theta^k)}     \bigg| ^2 dx \bigg)^{1/2}  \\ 
&\leq  \frac{\theta^k }{\varepsilon_0}\int^{\theta^k}_0\frac{\omega_\mathbf{A}(t,B(0,2))+C_1\omega_h(t,B(0,2))}{t}dt.
\end{split}
\end{equation*}
By \eqref{eq5151mon} and \eqref{eq5232tue} with $\theta\in(0,1/2)$, we observe that
\begin{align*}
\abs{\mathbf{B}(\varepsilon,k)}&\leq C\big(1+\sum^{k-1}_{i=0}\int^{\theta^i}_0\frac{\omega_\mathbf{A}(t,B(0,2))+C_1\omega_f(t,B(0,2))}{t}dt\big)\\
&\leq C\big\{ 1+C(1+\varepsilon_0)k \big\}\leq C\bigg\{ 1+C(1+\varepsilon_0)\big(1+\frac{\log(\varepsilon_0/\varepsilon)}{\log(1/\theta)} \big) \bigg\}\\
&\leq C(1+\frac{1}{\varepsilon}).
\end{align*}
Since $\chi^\beta_j$ is bounded in $L^2([0,1)^n;\bR^{n\times m^2})$ and also have local $L^\infty$ estimates with Dini mean oscillation of $\mathbf{A}$ (See, \cite[Theorem 1.5.]{DK17}), by \eqref{eq5232tue}, we have
\[
\fint_{B(0,\theta^k)}\abs{I^\beta_j(x)+\varepsilon \chi^\beta_j(x/\varepsilon)}^2 dx\leq C(\theta^k+\varepsilon)^2\leq C\varepsilon^2
\]
By combining all of them, we have
\begin{align}
&\bigg(\fint_{B(0,\theta^k)}\abs{v_\varepsilon(x)-\overline{v_\varepsilon}^{B(0,\theta^k)}}^2 dx\bigg)^{1/2}  \nonumber\\
&\leq  C\Bigg\{\frac{\varepsilon}{\varepsilon^2_0} \int^{\varepsilon/\varepsilon_0}_0\frac{\omega_\mathbf{A}(t,B(0,2))+C_1\omega_h(t,B(0,2))}{t}dt+1+\varepsilon\Bigg\}  \nonumber \\
&\leq C(1+\varepsilon). \label{eq5233tue}
\end{align}

For $x\in B(0,\theta/\varepsilon_0)$ with \eqref{eq5231tue}, let
\[
w_\varepsilon(x)=\frac{v_\varepsilon(\varepsilon x)-\overline{v_\varepsilon}^{B(0,\varepsilon\theta/\varepsilon_0)}}{1+\varepsilon}\quad\text{and}\quad \tilde{h} =\frac{\varepsilon}{1+\varepsilon}(h(\varepsilon x)-h(0)).
\]
It is easy to check that
\[
\cL_1 w_\varepsilon=\di \tilde{h} \;\text{in}\;B(0,\theta/\varepsilon_0),
\]
and 
\[
\int^1_0\frac{\omega_{\tilde{h} }(t,B(0,2))}{t}dt= \frac{\varepsilon}{1+\varepsilon}\int^1_0\frac{\omega_{h}(\varepsilon t,B(0,2))}{t}dt\leq\int_0^\varepsilon\frac{\omega_{h}(  t,B(0,2))}{t}dt<+\infty.
\]
By \eqref{eq5233tue} and \eqref{eq5232tue}, we have
\begin{align*}
\bigg(\fint_{B(0, \theta/\varepsilon_0)}\abs{w_\varepsilon(x)}^2 dx\bigg)^{1/2}&=\frac{1}{1+\varepsilon}\bigg(\fint_{B(0,\varepsilon \theta/\varepsilon_0)}\abs{v_\varepsilon(x)-\overline{v_\varepsilon}^{B(0,\varepsilon\theta/\varepsilon_0)}}^2 dx\bigg)^{1/2}\\
&\leq \frac{\theta^{kn}}{(1+\varepsilon)(\varepsilon\theta/\varepsilon_0)^n}\bigg(\fint_{B(0,\theta^k)}\abs{v_\varepsilon(x)-\overline{v_\varepsilon}^{B(0,\theta^k)}}^2 dx\bigg)^{1/2}\\
&\leq C.
\end{align*}

From $C^1$ estimates with Dini mean oscillation of $\mathbf{A}$ (See, \cite[Theorem 1.5.]{DK17}) and  Caccioppoli's inequality, we have 
\begin{align*}
\norm{\nabla w_\varepsilon}_{L^\infty(B(0,\theta/2))}&\leq C\fint_{B(0,\theta)}\abs{\nabla w_\varepsilon}+C\int^\theta_0\frac{\omega_{\tilde{h} }(t,B(0,\theta))}{t}dt\\
&\leq C\bigg( \fint_{B(0,\theta )}\abs{\nabla w_\varepsilon}^2\bigg)^{1/2}+C\\
&\leq \frac{C}{\theta/\varepsilon_0-\theta }\bigg( \fint_{B(0,\theta/\varepsilon_0)}\abs{ w_\varepsilon}^2\bigg)^{1/2}+\norm{\tilde{h} }_{L^\infty(B(0,\theta/\varepsilon_0))}+C\\
&\leq \frac{\varepsilon  C_1 }{1+\varepsilon}\int^1_0\frac{\omega_{h}(t,B(0,\varepsilon\theta/\varepsilon_0)}{t}dt+C\leq C.
\end{align*}
Thus, we conclude that
\[
\norm{\nabla u_\varepsilon}_{L^\infty(B(0,\varepsilon\theta/2))}\leq C(1+\varepsilon)J\leq C\Bigg\{  \bigg(\fint_{B(0,2)}\abs{  u_\varepsilon}^2 \bigg)^{1/2} +\int^1_0\frac{\omega_f(t,B(0,2))}{t}dt  \Bigg\},
\]
which, by translation, implies \eqref{lip_esti}.
\end{proof}

Finally, using scaling argument, we can prove the interior Lipschitz estimates (Theorem \ref{thm00}) from Lemma \ref{lemma03} and omit the detail of the proof.

\section{Proof of boundary Lipschitz estimates} \label{section4}
In this section, we introduce the Dirichlet correctors and their properties and prove the boundary Lipschitz estimate for \eqref{main_eq}.

\subsection{Green's function and Dirichlet correctors}
According to the results of \cite{HK07,TKB13}, the following results can be proved from Lemma \ref{lemmaholder} which is observed interior and boundary H\"{o}lder estimate. The details are omitted and can be seen in \cite{ShenBook}.

\begin{lemma}  \label{green}
Let $\Omega$ be a bounded $C^{1}$ domain in $\bR^n$ .
Assume the coefficients $\mathbf{A}=(a^{ij})$ of the operator $\cL_\varepsilon$ satisfy the  condition \eqref{ellipticity}, \eqref{perodicity} and are of vanishing mean oscillation in $\bR^n$ ($\textsf{VMO}_{\mathbf{A}}$). Then, there exists a unique Green's function (matrix) $G_\varepsilon(x,y)=(G^{ij}_\varepsilon(x,y))^m_{i,j=1}$ ($x\neq y$) which is continuous in $\{ (x,y)\in \Omega\times\Omega~:~x\neq y\}$ and satisfies the following estimates:
\begin{align}
\begin{split}\label{eq12051thu}
\abs{G_\varepsilon(x,y)} \leq C\abs{x-y}^{2-n}\quad&\hbox{if $\forall x,~y\in\Omega,~x\neq y,~n\geq3$,}  \\
\abs{G_\varepsilon(x,y)} \leq C(1+\log(\operatorname{diam}\Omega/\abs{x-y}))\quad&\hbox{if $\forall x,~y\in\Omega,~x\neq y,~n=2$,}  
\end{split}
 \\
\abs{G_\varepsilon(x,y)-G_\varepsilon(z,y)} \leq \frac{C\abs{x-z}^{\gamma_1}}{\abs{x-y}^{n-2+\gamma_1}}\quad&\hbox{if $\abs{x-z}<\frac{1}{4}\abs{x-y}$,}\label{eq12053thu}  
\end{align}
where $0<\gamma_1,~\gamma_2<1$ and $d_x:=\dist (x,\partial \Omega)$. The constants $C$ depends on $n$, $m$, $\lambda$, $\Lambda$, $\Omega$, $\textsf{VMO}_\mathbf{A}$ and , if necessary, $\gamma_1$, $\gamma_2$. Moreover, for $F\in L^p(\Omega;\bR^m)$ with $p>n/2$,
\begin{align}\label{eq08121mon}
u^\alpha_\varepsilon(x)=\int_{\Omega}G^{\alpha\beta}_\varepsilon(x,y)F^\beta(y)dy,
\end{align}
satisfies $\cL_\varepsilon u_\varepsilon=F$ in $\Omega$ and $u_\varepsilon=0$ on $\partial \Omega$.
\end{lemma}
Next, we introduce the Dirichlet corrector $\Phi_\varepsilon$, which replaces the (first-order) corrector $\chi$ near the boundary and vanishes at the (Dirichlet) boundary. In particular, it plays an essential role in studying boundary Lipschitz estimates for the Dirichlet problem.

 We define the Dirichlet corrector $\Phi_\varepsilon=(\Phi^\beta_{\varepsilon,i})=(\Phi^{\alpha\beta}_{\varepsilon,i})$ with $1\leq i \leq n$, $1\leq \alpha, \beta\leq m$, for the operator $\cL_\varepsilon$ in $\Omega$ by 
 \begin{equation} \label{eq7191wed}
 \cL_\varepsilon  \Phi^\beta_{\varepsilon,j}=0 \;\text{ in }\;\Omega,\quad \Phi^\beta_{\varepsilon,j}=I^\beta_j\;\text{ on }\;\partial \Omega.
 \end{equation} 
Here, the function $\Phi-I$ plays a role similar to $\varepsilon \chi(x/\varepsilon)$ for interior Lipschitz estimates and satisfies 
\begin{equation*}  
\cL_\varepsilon \big\{ \Phi^\beta_{\varepsilon,j} -I^\beta_j \big\} =\cL_\varepsilon \big\{ \varepsilon \chi^\beta_j (x/\varepsilon) \big\}.
\end{equation*}

The following lemma is about one of the various properties of the Dirichlet corrector, and other properties can be found in \cite{AL87,DLZ21,ShenBook}.

\begin{lemma} \label{bdry_corrector}
Let $\Omega$ be a bounded $C^{1}$ domain in $\bR^n$ .
Assume the coefficients $\mathbf{A}=(a^{ij})$ of the operator $\cL_\varepsilon$ satisfy the  condition \eqref{ellipticity}, \eqref{perodicity} and are of Dini mean oscillation in $\bR^n$. Let $\Phi_\varepsilon \in H^1(\Omega;\bR^{n\times m^2})$ be the solution of \eqref{eq7191wed}. Then, for any $\tau\in(0,1)$, there exists $C$ depending only one $n$, $m$, $\lambda$, $\Lambda$, $\Omega$, $\omega_\mathbf{A}$ and $\tau$ such that
\begin{equation} \label{eq5252fri}
\abs{\Phi^\beta_{\varepsilon,j}-I^\beta_j }\leq C \varepsilon^{1-\tau} d_x^\tau, \quad\hbox{$x\in \Omega$,}
\end{equation}
where $1\leq j\leq n$, $1\leq \beta\leq m$, and $d_x=\dist(x,\partial\Omega)$. In fact, we also have that
\begin{equation*} 
\norm{\Phi_\varepsilon}_{L^\infty(\Omega)}\leq C_3
\end{equation*}  
where $C_3=C_3(n, m, \lambda, \Lambda, \Omega,\omega_\mathbf{A})$.
\end{lemma}

\begin{proof}
Let $v_\varepsilon=(v^\beta_{\varepsilon,j})=\Phi^\beta_{\varepsilon,j}-I^\beta_j$, where $1\leq j\leq n$, $1\leq \beta\leq m$ and recall that $v_\varepsilon$ is a solution of $\cL_\varepsilon v_\varepsilon= -\di (\mathbf{A}(x/\varepsilon)\nabla \chi(x/\varepsilon))$ in $\Omega$ and $v_\varepsilon=0$ on $\partial \Omega$. From \eqref{corrector}, the $H^1$ estimate, and the $C^1$ estimates with Dini mean oscillation of $\mathbf{A}$  (See, \cite[Theorem 1.5.]{DK17}), we have that
\begin{align}\label{eq08123mon}
\norm{\nabla \chi}_{L^\infty([0,1)^n)}\leq C,
\end{align}
where $C=C(n,m,\lambda,\Lambda,\Omega,\omega_\mathbf{A})$. To show estimate \eqref{eq5252fri}, note that, from \cite[Theorem 5.4.2.]{ShenBook},
\begin{equation}\label{eq08122mon}
\int_\Omega\abs{\nabla_yG_\varepsilon (x,y)}d_y^{\tau -1}dy \leq Cd_x^\tau,
\end{equation}
for any $\tau\in (0,1)$, where $C=C(n,m,\lambda,\Lambda,\Omega,\omega_\mathbf{A})$. It follows from \eqref{eq08121mon}, \eqref{eq08123mon}, and \eqref{eq08122mon} that, for $x\in \Omega$,
\begin{align*}
\abs{v_\varepsilon(x)}&\leq C \int_\Omega \abs{\nabla_y G_\varepsilon(x,y)} dy\\
&\leq   C \int_\Omega \abs{\nabla_y G_\varepsilon(x,y)} [d_y\leq \varepsilon] + C \int_\Omega \abs{\nabla_y G_\varepsilon(x,y)} [d_y> \varepsilon]\\
&\leq C  \varepsilon^{1-\tau} d_x^\tau
\end{align*}
where we used Iverson braket $[\cdot]$. This implies the first estimate in the lemma. Similarly, the second estimate of the lemma can be obtained.
\end{proof}

\subsection{Compactness method via Dirichlet correctors}
In this subsection, we prove the boundary Lipschitz estimate with $C^{1,Dini}$ boundary data term in $C^{1,Dini}$ domains using by compactness method.

Recall that $\Omega$ is a $C^{1,Dini}$ domain if for each $x_0\in \partial \Omega$, there exists $R>0$ (independent of $x_0$) and a $C^{1,Dini}$ function $\psi:\bR^{n-1}\rightarrow \bR$ such that in a new coordinate system, $x_0$ becomes the origin,
\[
\Omega(0,R):=\Omega_\psi(0,R)=\{x\in B(0,R):x_n> \psi(x_1,\ldots,x_n)\}\;\;\text{and}\;\;\psi(0)=0.
\]

Note that $\Omega$ also satisfies the following condition: For any $x\in\overline{\Omega}$, 
\begin{equation} \label{condition_A}
|\Omega(x,R)|\geq C_\Omega |B(x,R)|,\;0<C_\Omega\leq 1\;\;\text{and}\;0<R<\operatorname{diam}\Omega.
\end{equation}

\begin{lemma}[\emph{One-step improvement}]  \label{lemma04}
Let the Dirichlet corrector $\Phi_{\varepsilon,j}^\beta$, denoted as $\Phi_{\varepsilon,j}^\beta(\cdot,\Omega(0,2),\mathbf{A})$, be defined   by \eqref{eq7191wed}. 
There exist constants $\varepsilon_0\in(0,1/4)$, $\theta\in(0,1/4)$, 
and $C_4>0$, depending only $n$, $m$, $\lambda$, $\Lambda$, $\omega_\mathbf{A}$, and $\Omega$ such that
if $u_\varepsilon$, $g$ satisfy
\begin{align*}
\bigg(\fint_{\Omega(0,2)}\abs{u_\varepsilon}^2 \bigg)^{1/2}\leq 1,\quad 
  \norm{\nabla g}_{L^\infty(\partial\Omega(0,2))}   \leq \varepsilon_0,\quad g(0)= \abs{\nabla g(0)}=0,
\end{align*}
with \eqref{ellipticity}, \eqref{perodicity}, and
\[ 
\cL_\varepsilon u_\varepsilon= 0 \;\text{ in }\;\Omega(0,2),\quad u_\varepsilon=g\;\text{ on }\;\partial \Omega(0,2),
\] 
then, for any $0<\varepsilon\leq\varepsilon_0$, 
\begin{equation}  \label{eq8021wed}
\bigg(\fint_{\Omega(0,\theta)} \bigg| u_\varepsilon(x)- \Phi^\beta_{\varepsilon,j}     \mathbf{B}^\beta_j(\varepsilon) \bigg| ^2 dx \bigg)^{1/2}   
\leq \theta\omega_\mathbf{A}(\theta,\Omega(0,2)).
\end{equation}
for some $\mathbf{B}(\varepsilon)=\big(\mathbf{B}^\beta_j(\varepsilon) \big)\in \bR^{n\times m}$ with the property that
\begin{align*}
\abs{\mathbf{B}(\varepsilon)}\leq \frac{C_4}{\theta}, \quad  \nabla \Phi^\beta_{\varepsilon,j}(0)\mathbf{B}^\beta_j(\varepsilon)=0.
\end{align*}
\end{lemma}
\begin{proof}
Let 
\[
\mathbf{B}^\beta_j(\varepsilon)=\mathbf{B}^\beta_j(\varepsilon,\Omega)=\mathbf{n}_j(0)\mathbf{n}_i(0) \overline{D_iu^\beta_\varepsilon}^{\Omega(0,\theta)}. 
\]
where $\mathbf{n}$ denotes the unit outward normal to $\partial\Omega(0,2)$. Then, it is easy to show that $\nabla \Phi^\beta_{\varepsilon,j}(0)\mathbf{n}_j(0)=\nabla I^\beta_j(0)\mathbf{n}_j(0)=0$. Also, by the boundary Caccioppoli's inequality, 
we have $\abs{\mathbf{B}(\varepsilon)}\leq C_4/\theta$.  

Suppose that \eqref{eq8021wed} is not true. Then there exist sequence $\{\varepsilon_l\}\subset (0,1/4)$, $\{\mathbf{A}_l\}\subset \mathsf{DMO}(\bR^n)$ satisfying \eqref{ellipticity}, \eqref{perodicity},   $\{ \nabla g_l \}\subset L^\infty(\partial\Omega(0,2);\bR^{m})$,  $\{ \psi_l \}\subset C^{1,Dini}(\bR^{n-1};\bR)$ and $\{  u_{\varepsilon_l}\}\subset W^{1,2}(B(0,2);\bR^m)$ such that $\varepsilon_l\rightarrow 0$,
\begin{align}
\bigg(\fint_{\Omega_{\psi_l}(0,2)}\abs{u_{\varepsilon_l}}^2 \bigg)^{1/2}\leq 1,\quad 
\norm{\nabla g_l}_{L^\infty(\partial \Omega_{\psi_l}(0,2))} \leq \varepsilon_l,\quad g_l(0)=\abs{\nabla g_l(0)}=0,\label{eq8023wed}\\
\cL_{\mathbf{A}_l,\varepsilon_l}u_{\varepsilon_l}=0\;\text{in}\;\Omega_{\psi_l}(0,2),   \quad u_{\varepsilon_l}=g_l\;\text{ on }\;\partial \Omega_{\psi_l}(0,2),\nonumber
\end{align}
and 
\begin{equation}  \label{eq8025wed}
\bigg(\fint_{\Omega_{\psi_l}(0,\theta)} \bigg| u_{\varepsilon_l}(x)- \Phi^{\beta,l}_{\varepsilon_l,j}     \mathbf{B}^{\beta}_j(\varepsilon_l)  \bigg| ^2 dx \bigg)^{1/2}   
>\theta\omega_\mathbf{A}(\theta,\Omega(0,2)),
\end{equation}
where suitable Dirichlet corrector $\Phi_{\varepsilon_l,j}^{\beta,l}=\Phi_{\varepsilon_l,j}^{\beta,l}(\cdot,\Omega_{\psi_l}(0,2),\mathbf{A}_l)$ and the constant $\mathbf{B}^\beta_j(\varepsilon_l,\Omega_\psi)$. Observe that, by \eqref{eq8023wed}, and boundary Caccioppoli's inequality, $\{u_{\varepsilon_l}\}$ is uniformly bounded in $W^{1,2}( \Omega_{\psi_l}(0,1))$. By passing to subsequences, we see that, as $l \rightarrow \infty$,
\begin{equation} \label{eq8111fri}
\begin{split}
&\psi_l\rightarrow \psi\quad\: \text{in}\; C^1(\abs{x'}<2),\qquad  \nabla g_l\rightarrow 0 \quad\;\text{in}\; L^\infty(\partial \Omega(0,1)) ,   
\qquad  \hat{\mathbf{A}_l}\rightarrow \mathbf{A}^0,\\
&u_{\varepsilon_l} \rightharpoonup  u_0\quad\;\text{in}\; L^2(\Omega(0,2)),\qquad D_j u_{\varepsilon_l} \rightharpoonup D_j u_0\quad\;\text{in}\; L^2(\Omega(0,1)),
\end{split}
\end{equation}
and 
\begin{equation} \label{eq8112fri}
\cL_{\mathbf{A}^0,0}u_{0}=0\;\text{in}\;\Omega(0,1),   \quad u_{0}=0\;\text{ on }\;\partial \Omega(0,1)
\end{equation}
where $\hat{\mathbf{A}_l}$ denotes the homogenized coefficients for $\mathbf{A}_l$ and $\mathbf{A}^0$ is some constant matrix satisfying \eqref{ellipticity}, \eqref{perodicity}.
We now let $l\rightarrow \infty$, in \eqref{eq8025wed} and \eqref{eq8111fri} with Lemma \ref{bdry_corrector}. This leads to 
\begin{equation}\label{eq8113fri}
\bigg( \fint_{\Omega(0,2)}\abs{u_0}^2\bigg)^{1/2}\leq 1
\end{equation}
and
\begin{equation*}  
\bigg(\fint_{\Omega(0,\theta)} \bigg| u_0(x)- x_j  \mathbf{B}_j(0) \bigg| ^2 dx \bigg)^{1/2}   
\geq \theta\omega_\mathbf{A}(\theta,\Omega(0,2)).
\end{equation*}
From boundary $C^2$ estimates of $\eqref{eq8112fri}$ with $\eqref{eq8113fri}$, for any $\theta\in(0,1/4)$, we have that
\[  
\bigg(\fint_{\Omega(0,\theta)} \bigg| u_0(x)-x_j  \mathbf{B}_j(0)    \bigg| ^2 dx \bigg)^{1/2} \leq C_4 \theta^2 
\]
where $C_4=C_4(n,m,\lambda,\Lambda)$. Indeed, for $x\in\Omega(0,\theta)$ with $|\nabla u_0(0)|=0$, we have
\begin{align*}
&\bigg| x_j \fint_{\Omega(0,\theta)}\frac{\partial u_0}{\partial x_j} -x_j\mathbf{B}_j(0)\bigg|\\
&\leq \bigg| x_j \mathbf{n}_i (0) \bigg\{ \mathbf{n}_i (0) \fint_{\Omega(0,\theta)}\bigg( \frac{\partial u_0}{\partial x_j}-\frac{\partial u_0}{\partial x_j}(0)  \bigg)- \mathbf{n}_j (0) \fint_{\Omega(0,\theta)}\bigg(\frac{\partial u_0}{\partial x_j} -\frac{\partial u_0}{\partial x_i}(0) \bigg)  \bigg\} \bigg|\\
&\leq C\theta^2 \norm{D^2 u_0}_{L^\infty(\Omega(0,1))}.
\end{align*}

Next, for some  $0<\varepsilon_0<\frac{1}{4}\min\{1,1/(\omega_\mathbf{A}(1,\Omega(0,2)))^2, C_\Omega/(1+C_1)\}$,  we can choose sufficiently small $\theta$ such that 
\begin{equation*}	 
C(\lambda,n,m)C_4\theta^{1-\beta}\leq \frac{1}{2} \quad\text{and}\quad 
 \theta+\int^{\theta}_0\frac{\omega_{\mathbf{A}}(t,\Omega(0,2))}{t}dt\leq \varepsilon_0,   
\end{equation*}
where $\theta<1/4$,   
and $0<\beta<1$ with $\eqref{dec_omega}$. It then follows that 
\[  
\bigg(\fint_{\Omega(0,\theta)} \bigg| u_0(x)- x_j  \mathbf{B}_j(0)   \bigg| ^2 dx \bigg)^{1/2} \leq \frac{1}{2}\theta \omega_\mathbf{A}(\theta,\Omega(0,2)),
\]
which is in contradiction. Therefore, $\eqref{eq8021wed}$ holds for some 
\[
  0<\varepsilon_0<\frac{1}{4} \min \bigg\{ 1,\frac{1}{(\omega_\mathbf{A}(1,\Omega(0,2)))^2}, \frac{C_\Omega}{1+C_1}  \bigg\}. 
\]
\end{proof}
 
We now turn to iteration estimates with $\varpi$, which describes all given integral condition.

\begin{lemma}[\emph{Iteration}]  \label{lemma05}
Let $\varepsilon_0$ and $\theta$ be the constants given by Lemma \ref{lemma04}. 
Suppose the functions $u_\varepsilon$ and $g$ satisfy
\begin{align}	\label{eq9061wed}
\bigg(\fint_{\Omega(0,2)}\abs{u_\varepsilon}^2 \bigg)^{1/2}\leq 1,\quad C_1\int^1_0\frac{\varrho_{\nabla  g}(t,\partial\Omega(0,2))}{t}dt \leq \varepsilon_0,\quad g(0)=|\nabla g(0)|= 0,  
\end{align}
and
\[ 
\cL_\varepsilon u_\varepsilon= 0 \;\text{ in }\;\Omega(0,2),\quad u_\varepsilon=g\;\text{ on }\;\partial \Omega(0,2),
\] 
then, for some $k\geq 1$ with $0<\varepsilon<\theta^{k-1}\varepsilon_0$, there exists constants $\mathbf{B}(\varepsilon,l)=\big(\mathbf{B}^\beta_j(\varepsilon,l)\big)\in \bR^{n\times m}$ for $0\leq l\leq k-1$, such that
\begin{equation}   \label{eq8171thu} 
\abs{\mathbf{B}(\varepsilon,l)}\leq  \  \frac{\tilde{C_4}}{ \theta}  ,\qquad \nabla\Pi^{\beta,l}_{\varepsilon,j}(0)\mathbf{B}^\beta_j(\varepsilon,l)=0
\end{equation}
and
\begin{equation}  \label{eq8172thu} 
\bigg(\fint_{\Omega(0,\theta^k)} \bigg| u_\varepsilon(x)-\sum^{k-1}_{l=0}\varpi(\theta^l,\Omega(0,2))  \Pi^{\beta,l}_{\varepsilon,j}(x)\mathbf{B}_j^\beta(\varepsilon,l)  \bigg| ^2 dx \bigg)^{1/2} \leq  \frac{\theta^k }{\varepsilon_0}\varpi(\theta^k,\Omega(0,2)) 
\end{equation}
where $\tilde{C_4}= C_4/({\varepsilon_0}^{3/2}\omega_\mathbf{A}(1,\Omega(0,2)))$,
\[
 \Pi^{\beta,l}_{\varepsilon,j}(\cdot)=\theta^l\Phi^\beta_{\frac{\varepsilon}{\theta^l}.j}(\theta^{-l}\cdot,\Omega_{\psi_{\theta^l}},\mathbf{A}),\quad\psi_{\theta^l}(x')=\theta^{-l}\psi(\theta^l x')
\]
and 
\begin{align*}
\varpi(s,\Omega(0,2)) &=\omega_\mathbf{A}(s,\Omega(0,2))+2\varrho_{\nabla g}(s,\partial \Omega(0,2))\\
&\qquad+\frac{\varrho_{\nabla \psi}(s,|x'|<2)}{\varrho_{\nabla \psi}(1,|x'|<2)[\varrho_{\nabla \psi}(1,|x'|<2)\neq 0]+[\varrho_{\nabla \psi}(1,|x'|<2) = 0]}
\end{align*}
with we used Iverson bracket $[\cdot]$ and any constant $C$.
\end{lemma}
\begin{proof}

We prove the lemma by induction on $k$. First, the case $k=1$ hold from Lemma \ref{lemma04} with 
\begin{align*}
 \mathbf{B}^\beta_j(\varepsilon,l)=\frac{\mathbf{B}^\beta_j(\varepsilon,\Omega)}{\varpi(1,\Omega(0,2))}. 
\end{align*}

Suppose there exists constants $\mathbf{B}(\varepsilon,l)$ such that \eqref{eq8171thu}, \eqref{eq8172thu} for all integers up to some $k \geq 1$, where $0\leq l\leq k-1$ with $\varepsilon <\theta^{k-1}\varepsilon_0$.
For $x\in \Omega_{\psi_{\theta^k}}(0,2)$, we consider the function
\begin{align*}
w_\varepsilon(x)=\frac{u_\varepsilon(\theta^k x)-\sum^{k-1}_{l=0}  \varpi(\theta^l,\Omega(0,2))\Pi^{\beta,l}_{\varepsilon,j}(\theta^k x)\mathbf{B}_j^\beta(\varepsilon,l) }{  \theta^k \varepsilon_0^{-1}  \varpi(\theta^k ,\Omega(0,2))}
\end{align*}
By the scaling property and \eqref{eq7191wed}, we have
\begin{equation} \label{eq10052thr}
\cL_{\varepsilon/\theta^k} w_\varepsilon=0 \;\text{ in }\;\Omega_{\psi_{\theta^k}}(0,2),\quad w_\varepsilon=\tilde{g} \;\text{ on }\;\partial\Omega_{\psi_{\theta^k}}(0,2),
\end{equation}
where 
\begin{align*}
\tilde{g} =\frac{g(\theta^k x)-\sum^{k-1}_{l=0}  \varpi(\theta^l,\Omega(0,2)) I^\beta_j(\theta^k x)\mathbf{B}^\beta_j(\varepsilon,l) }{  \theta^k \varepsilon_0^{-1}  \varpi(\theta^k,\Omega(0,2))} .
\end{align*}

It then follows from  \eqref{eq9061wed}, \eqref{eq8171thu} that
\begin{align*} 
\tilde{g} (0)=|\nabla \tilde{g} (0)|=0,
\end{align*}
and
\begin{align*}
&\norm{\nabla \tilde{g} }_{L^\infty(\partial\Omega_{\psi_{\theta^k}}(0,2))} \\&\leq  \frac{\varepsilon_0}{\varpi(\theta^k,\Omega(0,2))}\Bigg( \norm{\nabla  \tilde{g} (\theta^k \cdot)-\nabla g(0)}_{L^\infty(\partial\Omega_{\psi_{\theta^k}}(0,2))}\\
&\qquad +\sum^{k-1}_{l=0}\varpi(\theta^l,\Omega(0,2)) \norm{ \nabla  I^\beta_j(\theta^k \cdot)\mathbf{B}^\beta_j(\varepsilon,l)- \nabla\Pi^{\beta,l}_{\varepsilon,j}(0)\mathbf{B}^\beta_j(\varepsilon,l)}_{L^\infty(\partial\Omega_{\psi_{\theta^k}}(0,2))}  \Bigg)\\
&\leq  \frac{\varepsilon_0}{\varpi(\theta^k,\Omega(0,2))}\Bigg( \varrho_{\nabla g}(\theta^k,\partial\Omega(0,2))+ C_1 \int^1_0\frac{\varpi(t,\Omega(0,2))}{t}dt\frac{\tilde{C_4}}{\theta}\varrho_{\nabla \psi}(\theta^k,|x'|< 2)\Bigg)\\
&\leq \frac{\varepsilon_0}{2}+C_1\int^1_0\frac{\varpi(t,\Omega(0,2))}{t}dt\frac{\tilde{C_4}}{\theta}\varrho_{\nabla \psi}(1,|x'|< 2)\\
&\leq  \frac{\varepsilon_0}{2}+C_1\int^1_0\frac{\varpi(t,\Omega(0,2))}{t}dt\frac{\tilde{C_4}}{\theta}\frac{\varrho_{\nabla \psi}(\theta,|x'|< 2)}{\theta^\beta}\\
&\leq  \frac{\varepsilon_0}{2}+C_1\int^1_0\frac{\varpi(t,\Omega(0,2))}{t}dt\frac{\tilde{C_4}}{\theta}\frac{C_1}{\theta^\beta}\int^\theta_0\frac{\varrho_{\nabla \psi}(t,|x'|< 2)}{t}dt
\end{align*}
where \eqref{dec_omega} is used in the second last line. Note now that, let $\psi_\delta(x')=\delta^{-1}\psi(\delta x')$ for $x'\in \bR^{n-1}$, $\psi_\delta(0)=0$, $\norm{\nabla \psi_\delta}_{L^\infty(\bR^{n-1})}=\norm{\nabla \psi}_{L^\infty(\bR^{n-1})}$, and
\begin{align*}
  \int^1_0\frac{\varrho_{\nabla \psi_\delta}(t,|x'|<2)}{t}dt=\int^1_0\frac{\varrho_{\nabla\psi}(\delta t,|x'|<2)}{t}dt=\int^\delta_0\frac{\varrho_{\nabla \psi}(t,|x'|<2)}{t}dt.
\end{align*}
Thus, we can make an initial dilation of the independent variables. So from here on, we assume that $\Omega_\psi$ is such that $\psi$ satisfies
\begin{align}\label{eq10051thr}
\begin{split}
&\Bigg(\frac{1}{\varrho_{\nabla \psi}(1,|x'|<2)[\varrho_{\nabla \psi}(1,|x'|<2)\neq 0]+[\varrho_{\nabla \psi}(1,|x'|<2) = 0]}\\
&\qquad\qquad\qquad+C_1\int^1_0\frac{\varpi(t,\Omega(0,2))}{t}dt\frac{\tilde{C_4}}{\theta}\frac{C_1}{\theta^\beta}\Bigg)\int^\theta_0\frac{\varrho_{\nabla \psi}(t,|x'|< 2)}{t}dt<\frac{\varepsilon_0}{2}.
\end{split}
\end{align}
This implies that
\begin{equation*} 
\norm{\nabla  \tilde{g} }_{L^\infty(\partial\Omega_{\psi_{\theta^k}}(0,2))} \leq \varepsilon_0.
\end{equation*}

By the induction hypothesis with \eqref{eq8172thu}, we obtain
\begin{equation*} 
\bigg(\fint_{\Omega_{\psi_{\theta^k}}(0,2)}\abs{w_\varepsilon}^2 \bigg)^{1/2}\leq 1.
\end{equation*}
Applying Lemma \ref{lemma04} with \eqref{eq10052thr}, $\varepsilon \theta^l\leq \varepsilon \theta^{-k+1}$, we have
\begin{equation}  \label{eq10061fri}
\bigg(\fint_{\Omega_{\psi_{\theta^k}}(0,\theta)} \bigg| w_\varepsilon(x)- \Phi^\beta_{\varepsilon/\theta^k,j} (x,\Omega_{\psi_{\theta^k}},\mathbf{A})    \mathbf{B}^\beta_j(\varepsilon) \bigg| ^2 dx \bigg)^{1/2}   
\leq \theta\omega_{\tilde{\mathbf{A}} }(\theta,\Omega_{\psi_{\theta^k}}(0,2)).
\end{equation}
where $\tilde{\mathbf{A}} =\mathbf{A}(\theta^k x)$ and
\[
\mathbf{B}^\beta_j(\varepsilon)= \mathbf{B}^\beta_j(\varepsilon,\Omega_{\psi_{\theta^k}})=\mathbf{n}_j(0)\mathbf{n}_i(0)   \fint_{\Omega_{\psi_{\theta^k}}(0,\theta)} D_iw^\beta_\varepsilon.  
\]
By \eqref{eq10061fri} with scaling, we have 
\begin{align*}
\bigg(\fint_{\Omega(0,\theta^{k+1})} \bigg| u_\varepsilon(x)-\sum^{k}_{l=0}\varpi(\theta^l,\Omega(0,2))  \Pi^{\beta,l}_{\varepsilon,j}(x)\mathbf{B}_j^\beta(\varepsilon,l)  \bigg| ^2 dx \bigg)^{1/2}   \\
\quad\leq  \frac{\theta^{k+1} }{\varepsilon_0}\varpi(\theta^k,\Omega(0,2))  \omega_{\tilde{\mathbf{A}} }(\theta,\Omega_{\psi_{\theta^k}}(0,2)) 
\end{align*}
where
\begin{align*}
 \mathbf{B}^\beta_j(\varepsilon,k)=\frac{\mathbf{B}^\beta_j(\varepsilon,\Omega_{\psi_{\theta^k}})}
{\varepsilon_0}
\end{align*}
Note that by definition of $\omega$ with \eqref{condition_A}, for   $x\in \Omega_{\psi_{\theta^k}}(0,2)\subset \Omega(0,2)$,  
\[
\omega_{\tilde{\mathbf{A}} }(\theta,x)\leq \frac{2}{C_\Omega}\omega_{\mathbf{A}}( \theta^{k+1},x).
\]
By chosen constants $\varepsilon_0$, $\theta$ and definition of $\varpi$, we observe that
\begin{align*}
 &\frac{\theta^{k+1} }{\varepsilon_0}\varpi(\theta^k,\Omega(0,2))  \omega_{\tilde{\mathbf{A}} }(\theta,\Omega_{\psi_{\theta^k}}(0,2))\\
 &\leq \frac{\theta^{k+1} }{\varepsilon_0}\frac{2}{C_\Omega}\omega_{\mathbf{A}}(\theta^{k+1},\Omega(0,2)) C_1\Bigg(\int^{\theta^k}_0 \frac{\omega_\mathbf{A}(t,\Omega(0,2))}{t}dt+2\int^{\theta^k}_0 \frac{\varrho_{\nabla g}(s,\partial \Omega(0,2))}{t}dt\\
& \quad+\frac{\int^{\theta^k}_0 \varrho_{\nabla \psi}(t,|x'|< 2)/tdt}{\varrho_{\nabla \psi}(1,|x'|<2)[\varrho_{\nabla \psi}(1,|x'|<2)\neq 0]+[\varrho_{\nabla \psi}(1,|x'|<2) = 0]}   \Bigg)  \\
&\leq  \frac{\theta^{k+1} }{\varepsilon_0}\varpi(\theta^{k+1},\Omega(0,2)) \frac{2C_1}{C_\Omega}(\varepsilon_0+\frac{2\varepsilon_0}{C_1}+\varepsilon_0)\leq\frac{\theta^{k+1} }{\varepsilon_0}\varpi(\theta^{k+1},\Omega(0,2)).
\end{align*}
By definition of $\Pi^{\beta,l}_{\varepsilon,j}$ and Caccioppoli's inequality for $w_\varepsilon$ with constnats $\theta$, $\varepsilon_0$, we obtain \eqref{eq8171thu} with $\tilde{C_4}= C_4/({\varepsilon_0}^{3/2}\omega_\mathbf{A}(1,\Omega(0,2)))$. By combining all of them, we conclude \eqref{eq8172thu}. This completes the induction.

\end{proof}

The last step in compactness method to obtain boundary estimates is simplified by separating it into two-step, both of which use the blow-up argument.

\begin{lemma}  \label{lemma06}
Assume the coefficients $\mathbf{A}=(a^{ij})$ of the operator $\cL_\varepsilon$ satisfy the  condition \eqref{ellipticity}, \eqref{perodicity} and are of Dini mean oscillation in $\bR^n$. Let $\Omega$ be a bounded $C^{1,Dini}$ domain. Suppose $u_\varepsilon\in W^{1,2}(\Omega(0,2);\bR^m)$ be a weak solution of
\[ 
\cL_\varepsilon u_\varepsilon= 0 \;\text{ in }\;\Omega(0,2),\quad u_\varepsilon=g\;\text{ on }\;\partial \Omega(0,2),
\] 
where $g\in C^{1,Dini}(\partial\Omega(0,2);\bR^{m})$. Then,  
\begin{equation*}
\begin{split}
&\bigg( \fint_{\Omega(0,r)}|u_\varepsilon(x)-g(0)-\sum^n_{j=1}\Phi^\beta_{\varepsilon,j}\frac{\partial g}{\partial x_j}(0)|^2 \bigg)^{1/2} \\
&\quad\leq Cr\Bigg\{   \bigg(\fint_{\Omega(0,2)}\abs{  u_\varepsilon}^2 \bigg)^{1/2} +g(0)+\norm{\nabla g}_{L^\infty(\partial\Omega(0,2))}+\int^1_0\frac{\varrho_{\nabla g}(t,\partial\Omega(0,2))}{t}dt  \Bigg\},
\end{split}
\end{equation*}
for any $0<r<1$, where Dirichlet corrector $\Phi_{\varepsilon,j}^{\beta}=\Phi_{\varepsilon,j}^{\beta}(\cdot,\Omega_{\psi}(0,2),\mathbf{A})$ and $C=C(n,m,\lambda,\Lambda,\omega_{\mathbf{A}},\Omega)$.
\end{lemma}
\begin{proof}
We may assume that $0<\varepsilon<\varepsilon_0\theta$, where $\varepsilon_0$, $\theta$ are the constants given by Lemmae \ref{lemma04}.  Because in the case of $\varepsilon_0\theta\leq \varepsilon$, it is clearly by boundary $C^1$ estimates with Dini mean oscillation of $\mathbf{A}$ (See, \cite[Theorem 1.5.]{DEK18}, third paragraph of \cite[p. 453]{DEK18}) and the observation that, for $\varepsilon\geq\varepsilon_0\theta$,
\begin{align*}
\int^1_0\frac{\omega_{\mathbf{A}(\cdot/\varepsilon)}(t,\Omega(0,2))}{t}dt&\leq \frac{2}{C_\Omega}\int^1_0\frac{\omega_\mathbf{A}(t/\varepsilon,\Omega(0,2))}{t}dt\\
&\leq C\int^{\frac{1}{\varepsilon_0\theta}}_0\frac{\omega_\mathbf{A}(t,\Omega(0,2))}{t}dt<+\infty.
\end{align*}
Suppose that
\[
\theta^{i+1}\leq \frac{\varepsilon}{\varepsilon_0}<\theta^i\quad\text{for some $i\geq 1$.}
\]
Define
\[
J=\bigg(\fint_{\Omega(0,2)}\abs{u_\varepsilon}^2 \bigg)^{1/2}+g(0)+C_3\norm{\nabla g}_{L^\infty(\partial\Omega(0,2))} +\frac{C_1}{\varepsilon_0} \int^1_0\frac{\varrho_{\nabla  g}(t,\partial\Omega(0,2))}{t}dt.
\]
For $x\in\Omega(0,2)$, let
\[
v_{\varepsilon}=J^{-1}\bigg\{u_\varepsilon(x)-u_\varepsilon(0)-\sum^n_{j=1}\Phi^\beta_{\varepsilon,j}\frac{\partial g}{\partial x_j}(0)\bigg\},~~\text{and }~~~~ h(x)=J^{-1}\bigg\{g(x)-g(0)-\sum^n_{j=1}I^\beta_{j}\frac{\partial g}{\partial x_j}(0)\bigg\} ,
\]
where $\Phi^\beta_{\varepsilon,j}=\Phi^\beta_{\varepsilon,j}(\cdot, \Omega(0,2),\mathbf{A})$. It is easy to check that
\begin{equation}\label{eq10171the}
\bigg(\fint_{\Omega(0,2)}\abs{v_\varepsilon}^2 \bigg)^{1/2}\leq 1,\quad C_1\int^1_0\frac{\varrho_{\nabla  h}(t,\partial\Omega(0,2))}{t}dt \leq \varepsilon_0,\quad h(0)=|\nabla h(0)|= 0,  
\end{equation}
and
\[ 
\cL_\varepsilon v_\varepsilon= 0 \;\text{ in }\;\Omega(0,2),\quad v_\varepsilon=h\;\text{ on }\;\partial \Omega(0,2).
\] 
Under the above settings, it is sufficient to prove following
\begin{equation}\label{eq10161mon}
\bigg( \fint_{\Omega(0,r)}|v_\varepsilon|^2 \bigg)^{1/2}\leq Cr
\end{equation}
where $C=C(n,m,\lambda,\Lambda,\omega_{\mathbf{A}},\Omega)$. Here, we can assume that $0<r<\theta$. Indeed, the case where $\theta\leq r<1$ is trivial from \eqref{condition_A}, \eqref{eq10171the}.

To prove \eqref{eq10161mon}, we distinguish two cases:

\textbf{Case I.} We consider the case that
\[
\frac{\varepsilon}{\varepsilon_0}\leq r<\theta.
\]
In this case, we can assume that $\theta^{k+1}\leq r <\theta^k$ for some $k=1,\ldots,i,$ and by using Lemma \ref{lemma05}, we can obtain the following 
\begin{align*}
\bigg( \fint_{\Omega(0,r)}|v_\varepsilon|^2 \bigg)^{1/2}& \leq C\bigg( \fint_{\Omega(0,\theta^k)}|v_\varepsilon|^2 \bigg)^{1/2} \\
&\leq C\bigg(\fint_{\Omega(0,\theta^k)} \bigg| v_\varepsilon(x)-\sum^{k-1}_{l=0}\varpi(\theta^l,\Omega(0,2))  \Pi^{\beta,l}_{\varepsilon,j}(x)\mathbf{B}_j^\beta(\varepsilon,l)  \bigg| ^2 dx \bigg)^{1/2} \\
&\quad +C\sum^{k-1}_{l=0}\bigg(\fint_{\Omega(0,\theta^k)} \bigg| \varpi(\theta^l,\Omega(0,2))  \Pi^{\beta,l}_{\varepsilon,j}(x)\mathbf{B}_j^\beta(\varepsilon,l)  \bigg| ^2 dx \bigg)^{1/2}\\
&\leq  C\frac{\theta^k }{\varepsilon_0}\varpi(\theta^k,\Omega(0,2)) +C\sum^{k-1}_{l=0}\varpi(\theta^l,\Omega(0,2)) \norm{\Pi^l_\varepsilon}_{L^\infty(\Omega(0,\theta^k))}\frac{\tilde{C_4}}{\theta}\\
&:= \text{(i)}+\text{(ii)}.
\end{align*}
 
It is easy to check that, by choice of $\theta$, \eqref{eq10051thr}, \eqref{eq10171the}, 
\[
\text{(i)}\leq C\frac{\theta^k}{\varepsilon_0}\bigg(C_1\varepsilon_0+2\varepsilon_0+ C_1\frac{\varepsilon_0}{2}\bigg)\leq C\theta^k.
\]
To prove (ii), recall that $I^\beta_j(x)=x_je^\beta$, it follows from Lemma \ref{bdry_corrector} with $x\in \Omega(0,\theta^k)$ for $l<k$ that
\begin{align*}
\big| \Pi^{\beta,l}_{\varepsilon,j}(x)\big|&\leq \big|\Pi^{\beta,l}_{\varepsilon,j}(x)-I^{\beta}_j(x)\big|+\big|I^{\beta}_j(x)\big|\\
&= \big|\theta^l \Phi^{\beta}_{\frac{\varepsilon}{\theta^l},j}(\theta^{-l}x,\Omega_{\psi_{\theta^l}},\mathbf{A})-\theta^lI^{\beta}_j(\theta^{-l}x)\big|+\big|I^{\beta}_j(x)\big|\\
&\leq C\theta^l \bigg(\frac{\varepsilon}{\theta^l} \bigg)^{1-\tau} \dist(\theta^{-l}x,\partial\Omega_{\psi_{\theta^l}}(0,2))+2\theta^k\\
&\leq C\theta^{l\tau}\varepsilon^{1-\tau}2\theta^{k-l}+2\theta^k  \\
&\leq C\theta^{l\tau}(\varepsilon_0\theta^l)^{1-\tau}2\theta^{k-l}+2\theta^k\qquad\text{by $ \varepsilon\leq  \varepsilon_0\theta^k\leq \varepsilon_0\theta^l$}\\
&\leq C\varepsilon_0^{1/2}\theta^k+2\theta^k \leq C\theta^k \qquad\text{by chooing $\tau=1/2$}.
\end{align*}
Then, we can estimate that, based on the above and similarly to (i), 
 \[
\text{(ii)}\leq  C\theta^k.
\]
Thus, by combining all of the, with $\theta^k\leq r/\theta$, we obtain the estimate \eqref{eq10161mon} for first case.

\textbf{Case II.} We consider the case that
\[
0<r<\frac{\varepsilon}{\varepsilon_0}.
\]
We use a blow-up method. For $x\in \Omega_{\psi_\varepsilon}(0,2/\varepsilon_0)$, let
\[
w_\varepsilon(x) =\frac{v_\varepsilon(\varepsilon x)}{\varepsilon}\quad\text{and}\quad \tilde{h} (x)=\frac{h(\varepsilon x)}{\varepsilon}.
\]
It is easy to check that
\[ 
\cL_1 w_\varepsilon= 0 \;\text{ in }\;\Omega_{\psi_\varepsilon}(0,2/\varepsilon_0),\quad w_\varepsilon=\tilde{h} \;\text{ on }\;\partial \Omega_{\psi_\varepsilon}(0,2/\varepsilon_0)
\] 
and
\begin{equation}   \label{eq10191thr}
\norm{\nabla\tilde{h}}_{L^\infty(\partial\Omega_{\psi_\varepsilon}(0,2/\varepsilon_0))}+\int^1_0\frac{\varrho_{\nabla  \tilde{h}}(t,\partial\Omega_{\psi_\varepsilon}(0,2/\varepsilon_0))}{t}dt<+\infty\quad\text{and}\quad \tilde{h}(0)=|\nabla \tilde{h}(0)|= 0.
\end{equation}

From \eqref{eq10191thr}, boundary $C^1$ estimates with Dini mean oscillation of $\mathbf{A}$ (See, [Theorem 1.5.]\cite{DEK18}, third paragraph of \cite[p. 453]{DEK18}), boundary Caccioppoli's inequality and, for $0<s<1/\varepsilon_0$,  we have 
\begin{align*}
&\bigg( \fint_{\Omega_{\psi_\varepsilon}(0,s)}|w_\varepsilon|^2 \bigg)^{1/2}  \leq\norm{w_\varepsilon-w_\varepsilon(0)}_{L^\infty(\Omega_{\psi_\varepsilon}(0,s))}\leq Cs\norm{\nabla w_\varepsilon}_{L^\infty(\Omega_{\psi_\varepsilon}(0,\frac{4}{3\varepsilon_0}))}\\
&\quad\leq Cs\Bigg\{ \bigg( \fint_{\Omega_{\psi_\varepsilon}(0,\frac{5}{3\varepsilon_0})}|\nabla w_\varepsilon|^2 \bigg)^{1/2}+\norm{\nabla\tilde{h}}_{L^\infty(\partial\Omega_{\psi_\varepsilon}(0,\frac{5}{3\varepsilon_0}))}+\int^1_0\frac{\varrho_{\nabla  \tilde{h}}(t,\partial\Omega_{\psi_\varepsilon}(0,\frac{5}{3\varepsilon_0}))}{t}dt \Bigg\}\\
&\quad \leq Cs\Bigg\{ C\varepsilon_0\bigg( \fint_{\Omega_{\psi_\varepsilon}(0,2/\varepsilon_0)}|w_\varepsilon|^2 \bigg)^{1/2}+\norm{\nabla\tilde{h}}_{L^\infty(\partial\Omega_{\psi_\varepsilon}(0,2/\varepsilon_0)}+C \Bigg\}\\
&\quad \leq Cs\Bigg\{ \bigg( \fint_{\Omega_{\psi_\varepsilon}(0,2/\varepsilon_0)}|w_\varepsilon|^2 \bigg)^{1/2}+C \Bigg\}.
\end{align*}
Then, by scaling and \eqref{eq10161mon} with for first case $r=2\varepsilon/\varepsilon_0$, we imply the estimate \eqref{eq10161mon} for second case as follows
\begin{align*}
\bigg( \fint_{\Omega(0,r)}|v_\varepsilon|^2 \bigg)^{1/2} & \leq C r \varepsilon^2 \Bigg\{ \frac{1}{\varepsilon}\bigg( \fint_{\Omega(0,2\varepsilon/\varepsilon_0)}|u_\varepsilon|^2 \bigg)^{1/2}+C \Bigg\}\\
&\leq C r\varepsilon^2 \Bigg\{ \frac{1}{\varepsilon}\frac{2\varepsilon C}{\varepsilon_0}+C \Bigg\}\leq Cr.
\end{align*}
Thus, we have completed the proof of \eqref{eq10161mon}.

\end{proof}

\begin{proposition} \label{prop02}
Assume the coefficients $\mathbf{A}=(a^{ij})$ of the operator $\cL_\varepsilon$ satisfy the  condition \eqref{ellipticity}, \eqref{perodicity} and are of Dini mean oscillation in $\bR^n$. Let $\Omega$ be a bounded $C^{1,Dini}$ domain. Suppose $u_\varepsilon\in W^{1,2}(\Omega(x_0,2R);\bR^m)$ be a weak solution of
\[ 
\cL_\varepsilon u_\varepsilon= 0 \;\text{ in }\;\Omega(x_0,2R),\quad u_\varepsilon=g\;\text{ on }\;\partial \Omega(x_0,2R),
\] 
where $g\in C^{1,Dini}(\partial\Omega(x_0,2R);\bR^{m})$, $x_0\in \partial \Omega$ and $0<R<\frac{1}{2}\diam \Omega$. Then, we have
\begin{align}\label{lip_esti_bdry}
\begin{split}
\norm{\nabla u_\varepsilon}_{L^\infty(\Omega(x_0,R))}
\leq &C\Bigg\{\frac{1}{R}   \bigg(\fint_{\Omega(x_0,2R)}\abs{  u_\varepsilon}^p \bigg)^{1/p} +\frac{1}{R}\norm{g}_{L^\infty(\partial\Omega(x_0,2R))}\\&\qquad+\norm{\nabla g}_{L^\infty(\partial\Omega(x_0,2R))}+\int^R_0\frac{\varrho_{\nabla g}(t,\partial\Omega(x_0,2R))}{t}dt  \Bigg\},
\end{split}
\end{align} 
where $C=C(n,m,\lambda,\Lambda,\omega_{\mathbf{A}},\Omega,p)$ .
\end{proposition}

\begin{proof}
Without of loss generality, by rescaling, we may assume that $x_0=0$ and $R=1$.
We suppose that $0<\varepsilon<\varepsilon_0\theta$, where $\varepsilon_0$, $\theta$ are the constants given by Lemma \ref{lemma04}. 

Let $x\in \Omega(0,1)$. Here, we can assume that $0<d_x=\dist(x,\partial \Omega )<\theta<1/4$. Indeed, the case where $\theta\leq d_x<1$ is trivial from  the interior Lipschitz estimate (Theorem \ref{thm00}) with \eqref{condition_A}. Moreover, Choose $\bar{x}\in \partial\Omega(0,2)$ such that $|x-\bar{x}|=d_x$ and $\Omega(\bar{x}, \frac{3}{2} d_x)\subset \Omega(0,2)$. Then, similar to before, we distinguish two cases.

\textbf{Case I.} We consider the case that
\[
\frac{\varepsilon}{ \varepsilon_0}\leq d_x  <\theta.
\]

  Since the interior Lipschitz estimate (Theorem \ref{thm00}) of $u_\varepsilon(\cdot)-u_\varepsilon(0)$ for $\cL_\varepsilon$, it follows from Lemma \ref{lemma06} with change coordinate that
\begin{align*}
&|\nabla u_\varepsilon(x)|\\
&\leq \frac{C}{d_x}     \bigg(\fint_{B(x,d_x/2)}\abs{  u_\varepsilon(\cdot)-u_\varepsilon(0)}^2 \bigg)^{1/2} \leq \frac{C}{d_x}     \bigg(\fint_{\Omega(\bar{x},\frac{3}{2}d_x)}\abs{  u_\varepsilon(\cdot)-u_\varepsilon(0)}^2 \bigg)^{1/2}\\
&\leq  \frac{C}{d_x} \bigg( \fint_{\Omega(\bar{x},\frac{3}{2}d_x)}|u_\varepsilon(x)-g(0)-\sum^n_{j=1}\Phi^\beta_{\varepsilon,j}\frac{\partial g}{\partial x_j}(0)|^2 \bigg)^{1/2} \\
&\qquad\qquad+ \frac{C}{d_x} \bigg( \fint_{\Omega(\bar{x},\frac{3}{2}d_x)}|\sum^n_{j=1}\Phi^\beta_{\varepsilon,j}\frac{\partial g}{\partial x_j}(0)|^2 \bigg)^{1/2} \\
&\leq C\Bigg\{  \bigg(\fint_{\Omega(0,2)}\abs{  u_\varepsilon}^2 \bigg)^{1/2} +\norm{g}_{L^\infty(\partial\Omega(x_0,2R))} +\norm{\nabla g}_{L^\infty(\partial\Omega(0,2))}+\int^1_0\frac{\varrho_{\nabla g}(t,\partial\Omega(0,2))}{t}dt  \Bigg\}\\
&\qquad\qquad +\frac{C\norm{\nabla g}_{L^\infty(\partial\Omega(0,2))}}{d_x}\bigg( \fint_{\Omega(\bar{x},\frac{3}{2}d_x)}| \Phi^\beta_{\varepsilon,j} |^2 \bigg)^{1/2}  .
\end{align*}
To prove the estimate for first case, it is enough to show that
\begin{equation} \label{eq11011wed}
\bigg( \fint_{\Omega(\bar{x},\frac{3}{2}d_x)}| \Phi^\beta_{\varepsilon,j} |^2 \bigg)^{1/2}  \leq C d_x.
\end{equation}
 
Recall that $\varepsilon\leq \varepsilon_0 d_x$ and $I^\beta_j(x)=x_je^\beta$. Then, it follows from Lemma \ref{bdry_corrector} with $\tau=1/2$ that 
\begin{align*}
&\bigg( \fint_{\Omega(\bar{x},\frac{3}{2}d_x)}| \Phi^\beta_{\varepsilon,j} |^2 \bigg)^{1/2} \\
&\leq C\Bigg\{ \bigg( \fint_{\Omega(\bar{x},\frac{3}{2}d_x)}| \Phi^\beta_{\varepsilon,j}-I^\beta_j |^2 \bigg)^{1/2}+\bigg( \fint_{\Omega(\bar{x},\frac{3}{2}d_x)}| I^\beta_{j} |^2 \bigg)^{1/2} \Bigg\}\\
&\leq C\varepsilon^{1/2} \bigg( \fint_{\Omega(\bar{x},\frac{3}{2}d_x)} d_z dz \bigg)^{1/2} +Cd_x
\leq Cd_x
\end{align*}
Thus, \eqref{lip_esti_bdry} is proved for $ \varepsilon / \varepsilon_0 \leq d_x  <\theta.$

\textbf{Case II.} We consider the case that
\[
0<d_x<\frac{\varepsilon}{  \varepsilon_0}.
\]
Note that $\varepsilon/\varepsilon_0<1/4$, and 
$
\Omega(\bar{x},\frac{3}{2}\frac{\varepsilon}{\varepsilon_0})\subset\Omega(0,2).
$
We use a blow-up method. For $x\in \Omega_{\psi_\varepsilon}(\bar{x},\frac{3}{2\varepsilon_0})$, let
\[
\tilde{u}_\varepsilon(x) =\frac{u_\varepsilon(\varepsilon x)-u_\varepsilon(0)}{\varepsilon}\quad\text{and}\quad \tilde{g}(x)=\frac{g(\varepsilon x)-g(0)}{\varepsilon}.
\]
It is easy to check that
\[ 
\cL_1 \tilde{u}_\varepsilon = 0 \;\text{ in }\;\Omega_{\psi_\varepsilon}(\bar{x},\frac{3}{2\varepsilon_0}),\quad \tilde{u}_\varepsilon=\tilde{g}\;\text{ on }\;\partial \Omega_{\psi_\varepsilon}(\bar{x},\frac{3}{2\varepsilon_0}).
\] 
and
\begin{equation*}   
\norm{\nabla\tilde{g}}_{L^\infty(\partial\Omega_{\psi_\varepsilon}(\bar{x},\frac{3}{2\varepsilon_0}))}+\int^1_0\frac{\varrho_{\nabla  \tilde{g}}(t,\partial\Omega_{\psi_\varepsilon}(\bar{x},\frac{3}{2\varepsilon_0}))}{t}dt<+\infty.
\end{equation*}

From above condition, boundary $C^1$ estimates with Dini mean oscillation of $\mathbf{A}$ (See, \cite[Theorem 1.5.]{DEK18}, third paragraph of \cite[p. 453]{DEK18}) and  boundary Caccioppoli's inequality, we have
\begin{align*}
&\norm{\nabla \tilde{u}_\varepsilon }_{L^\infty (\Omega_{\psi_\varepsilon}(\bar{x},\frac{9}{8\varepsilon_0}))}\\
&\leq C\Bigg\{  \bigg(\fint_{\Omega_{\psi_\varepsilon}(\bar{x},\frac{5}{4\varepsilon_0})}\abs{  \nabla \tilde{u}_\varepsilon}\bigg) +\norm{\nabla \tilde{g}}_{L^\infty(\partial \Omega_{\psi_\varepsilon}(\bar{x},\frac{5}{4\varepsilon_0}))}+\int^1_0\frac{\varrho_{\nabla \tilde{g}}(t,\partial \Omega_{\psi_\varepsilon}(\bar{x},\frac{5}{4\varepsilon_0}))}{t}dt  \Bigg\}\\
&\leq C\Bigg\{  \bigg(\fint_{\Omega_{\psi_\varepsilon}(\bar{x},\frac{3}{2\varepsilon_0})}\abs{  \tilde{u}_\varepsilon}^2\bigg)^{1/2} +\norm{\nabla \tilde{g}}_{L^\infty(\partial \Omega_{\psi_\varepsilon}(\bar{x},\frac{3}{2\varepsilon_0}))}+\int^1_0\frac{\varrho_{\nabla \tilde{g}}(t,\partial \Omega_{\psi_\varepsilon}(\bar{x},\frac{3}{2\varepsilon_0}))}{t}dt  \Bigg\}
\end{align*}
where $C=C(n,m,\lambda,\Lambda,\omega_{\mathbf{A}},\Omega)$.
Since above inequality with rescaling,  Lemma \ref{lemma06} with change coordinate, it follows for $x\in\Omega(\bar{x},\frac{9\varepsilon}{8\varepsilon_0})$ that
\begin{align*}
&|\nabla u_\varepsilon(x)|\\
&\leq  C\Bigg\{ \frac{1}{\varepsilon} \bigg(\fint_{\Omega(\bar{x},\frac{3\varepsilon}{2\varepsilon_0})}\abs{   u_\varepsilon(\cdot)-  u_\varepsilon(0)}^2\bigg)^{1/2} +\norm{\nabla g}_{L^\infty(\partial \Omega (\bar{x},\frac{3\varepsilon}{2\varepsilon_0}))}+\int^1_0\frac{\varrho_{\nabla g}(t,\partial \Omega (\bar{x},\frac{3\varepsilon}{2\varepsilon_0}))}{t}dt  \Bigg\}\\
& \leq \frac{C}{\varepsilon} \bigg( \fint_{\Omega(\bar{x},\frac{3\varepsilon}{2\varepsilon_0})}|u_\varepsilon(x)-g(0)-\sum^n_{j=1}\Phi^\beta_{\varepsilon,j}\frac{\partial g}{\partial x_j}(0)|^2 \bigg)^{1/2} 
 \\
&~~~~+ \frac{C}{\varepsilon} \bigg( \fint_{\Omega(\bar{x},\frac{3\varepsilon}{2\varepsilon_0})}|\sum^n_{j=1}\Phi^\beta_{\varepsilon,j}\frac{\partial g}{\partial x_j}(0)|^2 \bigg)^{1/2}+ C\Bigg\{  \norm{\nabla g}_{L^\infty(\partial\Omega(0,2))}+\int^1_0\frac{\varrho_{\nabla g}(t,\partial\Omega(0,2))}{t}dt  \Bigg\}\\
&\leq C\Bigg\{  \bigg(\fint_{\Omega(0,2)}\abs{  u_\varepsilon}^2 \bigg)^{1/2} +\norm{g}_{L^\infty(\partial\Omega(x_0,2R))} +\norm{\nabla g}_{L^\infty(\partial\Omega(0,2))}+\int^1_0\frac{\varrho_{\nabla g}(t,\partial\Omega(0,2))}{t}dt  \Bigg\}\\
&\qquad\qquad +\frac{C\norm{\nabla g}_{L^\infty(\partial\Omega(0,2))}}{\varepsilon}\bigg( \fint_{\Omega(\bar{x},\frac{3\varepsilon}{2\varepsilon_0})}| \Phi^\beta_{\varepsilon,j} |^2 \bigg)^{1/2}.
\end{align*}
Simliar to \eqref{eq11011wed}, we obtain that
\begin{equation*} 
\bigg( \fint_{\Omega(\bar{x},\frac{3\varepsilon}{2\varepsilon_0})}| \Phi^\beta_{\varepsilon,j} |^2 \bigg)^{1/2}  \leq C \varepsilon.
\end{equation*}
Thus, \eqref{lip_esti_bdry} is proved for $0<d_x<\varepsilon/\varepsilon_0$ from by above two inequalities with standard (convexity) argument.

\end{proof}

\section{Uniform Lipschitz estimates} \label{section5}
In this section, we establish improved estimates for Green's function $G$ and turn to uniformly Lipschitz estimates for the main theorem.

\begin{theorem} \label{thm02}
Let $\Omega$ be a bounded $C^{1,Dini}$ domain in $\bR^n$. Assume the coefficients $\mathbf{A}=(a^{ij})$ of the operator $\mathcal{L}_\varepsilon$ satisfy the codition \eqref{perodicity}, \eqref{ellipticity} and are of Dini mean oscillation in $\mathbb{R}^n$. Then, the Green's function (matrix) $G_\varepsilon(x,y)=(G^{ij}_\varepsilon(x,y))^m_{i,j=1}$  for $x,~y\in\Omega$ $(x\neq y)$
satisfies
\begin{align}
|\nabla_x G_\varepsilon(x,y)|+&|\nabla_y G_\varepsilon(x,y)|\leq C|x-y|^{1-n} \label{eq12061wed}\\
|\nabla_x \nabla_y G_\varepsilon(x,y)|&\leq C|x-y|^{-n} \label{eq12062wed}      
\end{align}
where $C=C(n,m,\lambda,\Lambda,\Omega,\omega_{\mathbf{A}})$.
\end{theorem}
\begin{proof}
For fixed $x,~y\in \Omega$ with $x\neq y$, let $R=\frac{1}{8}\abs{x-y}$ and recall that $G^{*ij}_{\varepsilon}(x,y)=G^{ji}_\varepsilon(y,x)$. By Theorem \ref{thm00} and Proposition \ref{prop02}, we see that
\begin{align}  \label{eq12064wed}
\norm{\nabla u_\varepsilon}_{L^\infty(\Omega(x,R))}\leq  \frac{C}{R}   \bigg(\fint_{\Omega(x,2R)}\abs{  u_\varepsilon}^2 \bigg)^{1/2} ,
\end{align}
where $u_\varepsilon$ satisfying $\cL_\varepsilon u_\varepsilon=0$ in $\Omega_{2R}(x)$ and $u_\varepsilon=0$ on $\partial \Omega_{2R}(x)$. We let $u=G_\varepsilon(\cdot,y)$ if $n\geq 3$ and $u=G_\varepsilon(\cdot,y)-G_\varepsilon(z,y)$ with $z\in \Omega_{2R}(x)$ if $n=2$.  Then \eqref{eq12061wed} follows from \eqref{eq12064wed}, \eqref{eq12051thu} and \eqref{eq12053thu}. Moreover, by letting $u=\nabla_yG_\varepsilon(\cdot, y)$, which satisfies \eqref{eq12064wed}, we can derive \eqref{eq12062wed} from both \eqref{eq12064wed} and \eqref{eq12061wed}.


\end{proof}

\subsection{Proof of Theorem \ref{thm01}}
Without of generality, we may assume that $u_\varepsilon(x_0)=g(x_0)=0$ for any fixed a point $x_0\in \partial \Omega$.
Let
\begin{align*}
v_{1,\varepsilon}(x)=\int_\Omega G_\varepsilon(x,y)F(y)dy\quad\text{and}\quad
v_{2,\varepsilon}(x)=-\int_\Omega \frac{\partial}{\partial y_i} G^{\alpha\beta}_\varepsilon(x,y)f^\beta_i(y)dy.
\end{align*}
Then, we observe that
\begin{align*}
\cL_\varepsilon v_{1,\varepsilon}=F \;\text{ in }\;\Omega,\quad  v_{1,\varepsilon}=0\;\text{ on }\;\partial \Omega,\\
\cL_\varepsilon v_{2,\varepsilon}=\di f \;\text{ in }\;\Omega,\quad  v_{2,\varepsilon}=0\;\text{ on }\;\partial \Omega.
\end{align*}
Moreover, we note that, by Theorem \ref{thm02},
\begin{align*}
\norm{\nabla v_{1,\varepsilon}}_{L^\infty(\Omega)}&\leq C\norm{F}_{L^p(\Omega)},\quad \text{for}\;p>n\\
\norm{\nabla v_{2,\varepsilon}}_{L^\infty(\Omega)}&\leq C\int_\Omega |\nabla_y G_\varepsilon(x,y)||f(y)-f(x)|dy\leq C\int^1_0\frac{\varrho_f(t)}{t}dt,
\end{align*}
where $C=C(n,m,\lambda,\Lambda,\Omega,\omega_{\mathbf{A}})$. Thus, by $v_{1,\varepsilon}$ and $v_{2,\varepsilon}$, we may also assume that $F=0$ and $f=0$ in Theorem \ref{thm01}. Then, we can conclude Theorem \ref{thm01} from Theorem \ref{thm00}, Proposition \ref{prop02}, and Lemma \ref{lemmaholder} with covering argument.
\qed
 

\end{document}